\documentclass[12pt]{amsart}
\usepackage{amsmath,amssymb, xypic,verbatim,amscd,color}
\usepackage{graphicx}
\usepackage{colordvi}

\numberwithin{equation}{section}

\newcommand\frg{\mathfrak{g}}

\newcommand\Res{\operatorname{Res}}

\newtheorem{theorem}{Theorem}[section]
\newtheorem{remark}[theorem]{ Remark}

\newtheorem{corollary}[theorem]{Corollary}

\newtheorem{proposition}[theorem]{Proposition}

\newtheorem{lemma}[theorem]{Lemma}

\newtheorem{definition}[theorem]{Definition}

\begin{document}
\title[Rank-level duality of Conformal Blocks for odd orthogonal Lie algebras  ]{Rank-level duality of Conformal Blocks for odd orthogonal Lie algebras in genus $0$.}
%\author{Prakash Belkale}

\author{Swarnava Mukhopadhyay}
\thanks{The author was partially supported by NSF grant  DMS-0901249.}

\address{Department of Mathematics\\ University of Maryland\\  College Park, MD 20742-4015}
%\email{belkale@email.unc.edu}
\email{swarnava@umd.edu}
\subjclass[2010]{Primary 17B67, 14H60, Secondary 32G34, 81T40}
\begin{abstract}  Classical invariants for representations of one Lie group can often be related to invariants of some other Lie group. Physics suggests that the right objects to consider for these questions are certain refinements of classical invariants known as conformal blocks. Conformal blocks appear in algebraic geometry as spaces of global sections of line bundles on moduli stacks of parabolic bundles on a smooth curve. Rank-level duality connects a conformal block associated to one Lie algebra to a conformal block for a different Lie algebra. In this paper, we prove a rank-level duality for $\mathfrak{so}(2r+1)$ conformal blocks on the pointed projective line which was suggested by T. Nakanishi and A. Tsuchiya.
\end{abstract}
\maketitle
\section{Introduction}

It has been known for a long time that invariant theory of $\operatorname{GL}_r$ and the intersection theory of Grassmannians are related. This relation gives rise to some interesting isomorphisms between invariants of $\operatorname{SL}_r$ and $\operatorname{SL}_s$ for some positive integer $s$. To make it precise recall that the irreducible polynomial representations of $\operatorname{GL}_r$ are indexed by $r$ tuples of integers $\lambda=(\lambda^1\geq \dots \geq \lambda^r\geq 0) \in \mathbb{Z}^r$. Let $V_{\lambda}$ denote the corresponding irreducible $\operatorname{GL}_r$-module. 

Consider $\lambda=(\lambda^1\geq \dots \geq \lambda^r\geq 0)$ an $r$ tuple of integers such that $\lambda^1 \leq s$. The set of all such $\lambda$'s is in bijection with $\mathcal{Y}_{r,s}$, the set of all Young diagrams with at most $r$ rows and $s$ columns. For $\lambda, \mu, \nu$ in $\mathcal{Y}_{r,s}$ such that $|\lambda| + | \mu| + |\nu|=rs$, we know that 
$$\dim_{\mathbb{C}}(V_{\lambda}\otimes V_{\mu}\otimes V_{\nu})^{\operatorname{SL}_r}=\dim_{\mathbb{C}}(V_{\lambda^T}\otimes V_{\mu^T}\otimes V_{\nu^T})^{\operatorname{SL}_s},$$ where $|\lambda|$ denote the number of boxes in the Young diagram of $\lambda$ and $\lambda^T$ denotes the transpose of the Young diagram of $\lambda$. The above is not only a numerical ``strange" duality but the vector spaces are canonically dual to each other (see ~\cite{BGMA, Bel2}).

Physics suggests that to understand the above kind of relation for other groups the correct objects to consider are certain refinements of the co-invariants known as conformal blocks. Consider a finite dimensional simple complex Lie algebra $\frg$, a Cartan subalgebra $\mathfrak{h}$ and a non-negative integer $\ell$ called the level. Let $\vec{\lambda}= (\lambda_1,\dots, \lambda_n)$ be an $n$ tuple of dominant weights of $\frg$ of level $\ell$. To $n$ distinct points $\vec{p}=(P_1,\dots, P_n)$ with coordinates $\vec{z}=(z_1,\dots, z_n)$ on $\mathbb{P}^1$, one associates a finite dimensional vector space $\mathcal{V}_{\vec{\lambda}}(\frg, \ell, \vec{z})$ known as the space of covacua. The dual of $\mathcal{V}_{\vec{\lambda}}(\frg, \ell, \vec{z})$ is called a conformal block and is denoted by $\mathcal{V}^{\dagger}_{\vec{\lambda}}(\frg, \ell, \vec{z})$. We refer the reader to Section 2 for more details. More generally, one can define conformal blocks associated to $n$ distinct points on curves of arbitrary genus with at most nodal singularities (see Section 2). Conformal blocks form a vector bundle on $\overline{\operatorname{M}}_{g,n}$, the moduli stack of stable $n$ pointed curves of genus $g$.

 Rank-level duality is a duality between conformal blocks associated to two different Lie algebras. In ~\cite{NT}, T. Nakanishi and A. Tsuchiya proved that on $\mathbb{P}^1$, certain conformal blocks of $\mathfrak{sl}(r)$ at level $s$ are dual to conformal blocks of $\mathfrak{sl}(s)$ at level $r$. In ~\cite{A}, T. Abe proved rank-level duality statements between conformal blocks of type $\mathfrak{sp}(2r)$ at level $s$ and $\mathfrak{sp}(2s)$ at level $r$. It is important to point out that there are no known relations between the classical invariants for the Lie algebras $\mathfrak{sp}(2r)$ and $\mathfrak{sp}(2s)$.

The rank-level duality of conformal blocks has a geometric perspective under the identification of conformal blocks with the space of non-abelian $G$-theta functions. This is known as strange duality. The strange duality conjecture for $\operatorname{SL}_n$ says that the space of generalized theta functions associated to the pairs $(p,q)$, $(q,p)$ are naturally dual to each other, the duality being induced from the tensor product of vector bundles. This conjecture was proved by P. Belkale (see ~\cite{Bel3, Bel}) and also by A. Marian and D. Oprea ~\cite{MO}. The symplectic strange duality conjecture in ~\cite{B1} was proved by T. Abe (see ~\cite{A}). For a survey of these results, we refer the reader to ~\cite{MO2, Pa2, Po}. Rank-level duality isomorphisms also give relations in the Picard group of the Grothendieck-Knudsen moduli space $\overline{\operatorname{M}}_{0,n}$ (see \cite{M3}).

%\section{Rank-level duality for odd orthogonal Lie algebras}\label{odd}
The paper ~\cite{NT} (Section 6, page 368) suggests that one can try to answer similar rank-level duality questions for orthogonal Lie algebras on $\mathbb{P}^1$. Furthermore, it is pointed out in ~\cite{NT} that one should only consider the tensor representations, i.e. representations that lift to representations of the special orthogonal group (see Section 6 in ~\cite{NT}). In the following, we answer the above question for odd orthogonal Lie algebras. 

Throughout this paper, we assume that $r,s \geq 3$. Let $P^0_{2s+1}(\mathfrak{so}(2r+1))$ denote the set of tensor representations of $\mathfrak{so}(2r+1)$ of level $2s+1$. We can realize the set $P^0_{2s+1}(\mathfrak{so}(2r+1))$ as a disjoint union of $\mathcal{Y}_{r,s}$ and $\sigma(\mathcal{Y}_{r,s})$, where $\sigma$ is an involution $P^0_{2s+1}(\mathfrak{so}(2r+1))$ that corresponds to the action of a diagram automorphism of the affine Lie algebra $\widehat{\mathfrak{so}}(2r+1)$ (see Section \ref{branching2}). Our main theorem is the following:
\begin{theorem}\label{main} Let $\vec{\lambda}=(\lambda_1,\dots, \lambda_n) \in \mathcal{Y}_{r,s}^n$ be an $n$ tuple of weights in $P^0_{2s+1}(\mathfrak{so}(2r+1))$. %Then the following conformal blocks are isomorphic.
\begin{enumerate}
\item If $\sum_{i=1}^n|\lambda_i|$ is even, then $$\mathcal{V}_{\vec{\lambda}}(\mathfrak{so}(2r+1), 2s+1, \vec{z})\simeq \mathcal{V}^{\dagger}_{\vec{\lambda}^T}(\mathfrak{so}(2s+1), 2r+1, \vec{z}),$$ where $\vec{z}$ is a tuple of $n$ distinct points on $\mathbb{P}^1$.

\item If $\sum_{i=1}^n|\lambda_i|$ is odd, then $$\mathcal{V}_{\vec{\lambda},0}(\mathfrak{so}(2r+1), 2s+1, \vec{z})\simeq \mathcal{V}^{\dagger}_{\vec{\lambda}^T,\sigma(0)}(\mathfrak{so}(2s+1), 2r+1, \vec{z}),$$ where $\vec{z}$ is a tuple of $(n+1)$ distinct points on $\mathbb{P}^1$.

\item If $\sum_{i=1}^n|\lambda_i|$ is even, then $$\mathcal{V}_{\vec{\lambda},\sigma(0)}(\mathfrak{so}(2r+1), 2s+1, \vec{z})\simeq \mathcal{V}^{\dagger}_{\vec{\lambda}^T,\sigma(0)}(\mathfrak{so}(2s+1), 2r+1, \vec{z}),$$ where $\vec{z}$ is a tuple of $(n+1)$ distinct points on $\mathbb{P}^1$.
\end{enumerate}

\end{theorem}
The details about formulations of the Theorem \ref{main} can be found in Section 6. 
\begin{remark}
The above statements are independent of each other. The three statements above along with Proposition \ref{diaaut2} and propagation of vacua (see Section 2.4) cover all possible rank-level duality maps for the pair $(\mathfrak{so}(2r+1),\mathfrak{so}(2s+1))$. In the first statement, all of the chosen weights for $\mathfrak{so}(2r+1)$ lie in $\mathcal{Y}_{r,s}$ and for $\mathfrak{so}(2s+1)$ lie in $\mathcal{Y}_{s,r}$. In the second statement all weights for $\mathfrak{so}(2r+1)$ are in $\mathcal{Y}_{r,s}$ and all except one are in $\mathcal{Y}_{s,r}$ for $\mathfrak{so}(2s+1)$.

Finally the third statement, though has similar conditions as the first statement is independent of it. Here all but one weight of $\mathfrak{so}(2r+1)$ are in $\mathcal{Y}_{r,s}$ and similarly for $\mathfrak{so}(2s+1)$. It might appear that one can use propagation of vacua (see Section 2.4) and diagram automorphisms (see Proposition 3.6) to get the third statement from the first, but this is not possible. The choice of weights between the first (the weight $0$ at level $2r+1$) and the third (the weight $\sigma(0)$ at level $2r+1$) statement differ only at one ( odd number) point and the diagram automorphism $\sigma$ is of order two. So we can not remove the $\sigma$ from the third statement keeping all the other weights in $\mathcal{Y}_{r,s}$.
\end{remark}

%\section{General set up of rank-level duality}\label{general}
We briefly discuss the general context of rank-level duality maps. We closely follow the methods used in ~\cite{A, BP, NT} but there are significant differences in key steps.  

Let $\frg_1$, $\frg_2$ and $\frg$ be simple Lie algebras and consider an embedding of Lie algebras $\phi :\frg_1\oplus \frg_2\rightarrow \frg$. We extend it to a map of affine Lie algebras $\widehat{\phi}:\widehat{\frg}_1\oplus \widehat{\frg}_2\rightarrow \widehat{\frg}$. Consider a level one integrable highest weight module $\mathcal{H}_{\Lambda}(\frg)$, and restrict it to $\widehat{\frg}_1\oplus \widehat{\frg}_2$. The module $\mathcal{H}_{\Lambda}(\frg)$ decomposes into irreducible integrable $\widehat{\frg}_1\oplus \widehat{\frg}_2$-modules of level $\ell=(\ell_1,\ell_2)$ in the following way:

$$\bigoplus_{(\lambda,\mu)\in B(\Lambda)}m_{\lambda, \mu}^{\Lambda}\mathcal{H}_{\lambda}(\frg_1)\otimes \mathcal{H}_{\mu}(\frg_2)\simeq \mathcal{H}_{\Lambda}(\frg),$$
 where $\ell=(\ell_1, \ell_2)$ is the Dynkin multi-index of $\phi$, $m_{\lambda, \mu}^{\Lambda}$ is the multiplicity of the component $\mathcal{H}_{\lambda}(\frg_1)\otimes \mathcal{H}_{\mu}(\frg_2)$ and $B(\Lambda)$ is an indexing set for the components. In general, the number of components $|B(\Lambda)|$ may be infinite. We only consider those embeddings such that $|B(\Lambda)|$ is finite. These embeddings are known as conformal embeddings (see ~\cite{K} for more details).

 %We extend $\phi$ to a map $\widehat{\phi}:\widehat{\frg}_1\oplus \widehat{\frg}_2 \rightarrow \widehat{\frg}$ of affine Lie algebras. Let $\ell=(\ell_1, \ell_2)$ be the Dynkin multi index of the embedding $\phi$. We restrict ourselves to those embeddings such that $\mathcal{H}_{\Lambda}(\frg)$ restricted to $\widehat{\frg}_1\oplus \widehat{\frg}_2$ breaks (Physicists call this branching)  up as a finite direct sum of irreducible $\widehat{\frg}$ modules $\mathcal{H}_{\lambda_i}(\frg_1)\otimes \mathcal{H}_{\mu_i}(\frg_2)$ of multiplicity $m_{\lambda, \mu}^{\Lambda}$. These embeddings are known as conformal embeddings. 
  Further assume that $m_{\lambda, \mu}^{\Lambda}=1$ for any level 1 weight $\Lambda$. Thus, for an $n$ tuple $\vec{\Lambda}=(\Lambda_1,\dots, \Lambda_n)$ of level one dominant weights of $\frg$, we have an injective map:
  $$\bigotimes_{i=1}^n(\mathcal{H}_{\lambda_i}(\frg_1)\otimes \mathcal{H}_{\mu_i}(\frg_2)) \rightarrow \bigotimes_{i=1}^n\mathcal{H}_{\Lambda_i}(\frg).$$ We consider a tuple of $n$ distinct points $\vec{z}$ on $\mathbb{P}^1$ and taking ``coinvariants", we get a map 
$$\alpha : \mathcal{V}_{\vec{\lambda}}(\frg_1, \ell_1, \vec{z})\otimes \mathcal{V}_{\vec{\mu}}(\frg_2, \ell_2, \vec{z})\rightarrow \mathcal{V}_{\vec{\Lambda}}(\frg, 1, \vec{z}),$$ 
where $\vec{\lambda}=(\lambda_1,\dots, \lambda_n)$ and $\vec{\mu}=(\mu_1,\dots, \mu_n)$. We refer the reader to Section 2 for a detailed description. If $\dim_{\mathbb{C}}(\mathcal{V}_{\vec{\Lambda}}(\frg, 1, \vec{z}))=1$, we get a map $ \mathcal{V}_{\vec{\lambda}}(\frg_1, \ell_1, \vec{z})\rightarrow  \mathcal{V}^{\dagger}_{\vec{\mu}}(\frg_2, \ell_2, \vec{z})$. This map is known as the rank-level duality map. The above analysis with the embedding ${\mathfrak{so}}(2r+1)\oplus {\mathfrak{so}}(2s+1)\rightarrow {\mathfrak{so}}((2r+1)(2s+1))$ give the maps considered in Theorem ~\ref{main}. In Section 3, we define rank-level duality for conformal blocks on $n$ pointed nodal curves of arbitrary genus.

\begin{remark}
Note that the conformal blocks in Theorem ~\ref{main} can be identified with the space of global sections of a line bundle on moduli stacks of $\operatorname{Spin}$-bundles over $\mathbb{P}^1$ with parabolic structures on marked points (see \cite{LS}), but we have not been able to define the rank-level duality map in Theorem ~\ref{main} geometrically. 
\end{remark} 
%\section{Idea of Proof}\label{ideas}
We now discuss the main body of the proof of Theorem ~\ref{main}. This can be broken up into several steps:

\subsubsection{Dimension Check} Using the Verlinde formula, we show that the dimensions of the source and the target of the conformal blocks in Theorem ~\ref{main} are the same. Unlike the case in ~\cite{A}, we do not have a bijection between $P_{2s+1}(\mathfrak{so}(2r+1))$ and $P_{2r+1}(\mathfrak{so}(2s+1))$. We get around the problem by considering bijection of the orbits of $P_{2s+1}(\mathfrak{so}(2r+1))$ and $P_{2r+1}(\mathfrak{so}(2s+1))$ under the involution $\sigma$ as described in ~\cite{OW}. Let $\vec{\lambda} \in \mathcal{Y}_{r,s}^n$ and $\Gamma=\{1,\sigma\}$ be the group of diagram automorphisms of $\widehat{\mathfrak{so}}(2r+1)$ acting on $P_{2s+1}(\mathfrak{so}(2r+1))$. The Verlinde formula in this case takes the form 

$$\sum_{\mu \in P_{2s+1}(\mathfrak{so}(2r+1))/\Gamma}f(\mu, \vec{\lambda})|\operatorname{Orb}_{\mu}|,$$ where $f(\mu, \vec{\lambda})$ is a function, constant on the orbit of $\mu$ and $|\operatorname{Orb}_{\mu}|$ denotes the cardinality of the orbit of $\mu$ under the action of $\Gamma$. Using a non-trivial trigonometric identity in ~\cite{OW} and a generalization of Lemma A.42 in ~\cite{FH}, we show that $f(\mu, \vec{\lambda})|\operatorname{Orb}_{\mu}|$ is same for the corresponding orbit for the Lie algebra $\mathfrak{so}(2s+1)$ at level $2r+1$. For details, we refer the reader to Section 7. 

\subsubsection{Flatness of rank-level duality} The rank-level duality map has constant rank when $\vec{z}$ varies (see ~\cite{Bel}). The conformal embedding is important in this case as it ensures that the rank-level duality map is flat with respect to the $KZ/Hitchin/WZW$ connection (see ~\cite{Bel}) on sheaves of vacua over any family of smooth curves. 

\subsubsection{Degeneration of a smooth family}Let $C_1\cup C_2$ be a nodal curve, where $C_1$ and $C_2$ are isomorphic to $\mathbb{P}^1$ intersecting at one point. A conformal block on $C_1\cup C_2$ is isomorphic to a direct sum of conformal blocks on the normalization of $C_1\cup C_2$. This property is known as factorization of conformal blocks. A key ingredient in the proof of rank-level duality in ~\cite{A} is the compatibility of the rank-level duality with factorization. T. Abe uses it to conclude that the rank-level duality map is an isomorphism on certain nodal curves. 

This property for nodal curves is no longer true for our present case due to the presence of ``non-classical" components (i.e. components that do not appear in the branching of finite dimensional irreducible modules) in the branching of highest weight integrable modules. We refer the reader to Section 3 for more details. 

We consider a family of smooth curves degenerating to a nodal curve $X_0$. Instead of looking at the nature of the rank-level duality map on the nodal curve, we study the nature of the rank-level duality map on nearby smooth curves of the nodal curve $X_0$ under any conformal embedding. We use the ``sewing procedure" of ~\cite{TUY} to understand the decomposition of the rank-level duality map near the nodal curve $X_0$. The methods used in this step are similar to ~\cite{BP}. This degeneration technique and the flatness of the rank-level duality enable us to use induction similar to ~\cite{A, NT} to reduce to the case for one dimensional conformal blocks on $\mathbb{P}^1$ with three marked points. As before presence of non-classical components bring in complications. Detailed description of the above can be found in Section \ref{mainproof}.

\subsubsection{Minimal Cases} We are now reduced to showing that rank-level duality maps for one dimensional conformal blocks on $\mathbb{P}^1$ with three marked points are non-zero. Our proof of this step differs significantly from that in ~\cite{A} as we were not able to use any geometry of parabolic vector bundles with a non-degenerate form. This is again due to the presence of non-classical components. Using ~\cite{H}, we construct explicit vectors $|v_1\otimes v_2\otimes v_3\rangle$ in the tensor product of three highest weight integrable modules (see Section 8) and show by using gauge symmetry (see Section 2) that $\langle\Psi |v_1\otimes v_2\otimes v_3\rangle \neq 0$, where $\langle \Psi|$ is a non-zero element of an one dimensional conformal block at level one for $\mathfrak{so}((2r+1)(2s+1))$. It will be very interesting if one can define rank-level duality maps purely using the language of vector bundles with a non-degenerate form.

\subsection{Acknowledgments}
I am very grateful to Prakash Belkale for the discussions and ideas that were immensely useful during the preparation of this manuscript. This paper will form a part of my dissertation at UNC-Chapel Hill. I would also like to thank Shrawan Kumar, Christian Pauly and Catharina Stroppel for helpful discussions.

\section{Basic Definitions in the theory of Conformal blocks}\label{conformalblock}

We recall some basic definitions from ~\cite{TUY} in the theory of conformal blocks. Let $\frg$ be a simple Lie algebra over $\mathbb{C}$ and $\mathfrak{h}$ a Cartan subalgebra of $\frg$. We fix the decomposition of $\frg$ into root spaces 
$$\frg=\mathfrak{h} \oplus \sum_{\alpha \in \Delta}\frg_{\alpha},$$ where $\Delta$ is the set of roots decomposed into a union of $\Delta_{+}\sqcup\Delta_{-}$ of positive and negative roots. Let $(,)$ denote the Cartan Killing form on $\frg$ normalized such that $(\theta, \theta)=2$, where $\theta $ is the longest root and we identify $\mathfrak{h}$ with $\mathfrak{h}^*$ using the form $(,)$. 

\subsection{Affine Lie algebras} We define the affine Lie algebra $\widehat{\frg}$ to be 
$$\widehat{\frg}:= \frg\otimes \mathbb{C}((\xi)) \oplus \mathbb{C}c,$$ where $c$ belongs to the center of $\widehat{\frg}$ and the Lie bracket is given as follows:
$$[X\otimes f(\xi), Y\otimes g(\xi)]=[X,Y]\otimes f(\xi)g(\xi) + (X,Y)\Res_{\xi=0}(gdf).c,$$ where $X,Y \in \frg$ and $f(\xi),g(\xi) \in \mathbb{C}((\xi))$. 

Let $X(n)=X\otimes \xi^n$ and $X=X(0)$ for any $X \in \frg $ and $n \in \mathbb{Z}$. The finite dimensional Lie algebra $\frg$ can be realized as a subalgebra of $\widehat{\frg}$ under the identification of $X$ with $X(0)$. 

\subsection{Representation theory of affine Lie algebras} The finite dimensional irreducible modules of $\frg$ are parametrized by the set of dominant integral weights $P_{+} \subset \mathfrak{h}^*$. Let $V_{\lambda}$ denote the irreducible module of highest weight $\lambda \in P_{+}$ and $v_{\lambda}$ denote the highest weight vector.

We fix a positive integer $\ell$ which we call the level. The set of dominant integral weights of level $\ell$ is defined as follows: 
$$P_{\ell}(\frg):=\{ \lambda \in P_{+} | (\lambda, \theta) \leq \ell\}.$$
For each $\lambda \in P_{\ell}(\frg)$ there is a unique irreducible integrable highest weight $\widehat{\frg}$-module $\mathcal{H}_{\lambda}(\frg)$ which satisfies the following properties: 
\begin{enumerate}
\item $V_{\lambda} \subset \mathcal{H}_{\lambda}(\frg),$
\item The central element $c$ of $\widehat{\frg}$ acts by the scalar $\ell$, 
\item Let $v_{\lambda}$ denote a highest weight vector in $V_{\lambda}$, then 
$$X_{\theta}(-1)^{\ell-(\theta, \lambda)+1}v_{\lambda}=0,$$ where $X_{\theta}$ is a non-zero element in the weight space of $\frg_{\theta}$. 
 Moreover, $\mathcal{H}_{\lambda}(\frg)$ is generated by $V_{\lambda}$ over $\widehat{\frg}$ with the above relation. When $\lambda=0$, the corresponding $\widehat{\frg}$-module $\mathcal{H}_{0}(\frg)$ is known as the vacuum representation.
\end{enumerate}

\subsection{Conformal blocks}We fix a $n$ pointed curve $C$ with formal neighborhood $\eta_1,\dots, \eta_n$ around the $n$ points $\vec{p}=(P_1,\dots, P_n)$, which satisfies the following properties :
\begin{enumerate}
\item The curve $C$ has at most nodal singularities,
\item The curve $C$ is smooth at the points $P_1,\dots, P_n$,
\item $C-\{P_1, \dots, P_n\}$ is an affine curve,
\item A stability condition (equivalent to the finiteness of the automorphisms of the pointed curve),
\item Isomorphisms $\eta_i:\widehat{\mathcal{O}}_{C,P_i}\simeq \mathbb{C}[[\xi_i]]$ for $i=1,\dots, n$. 
\end{enumerate}
We denote by $\mathfrak{X}=(C;\vec{p}; \eta_1,\dots,\eta_n)$ the above data associated to the curve $C$. We define another Lie algebra
$$\widehat{\frg}_n:=\bigoplus_{i=1}^n\frg\otimes_{\mathbb{C}}\mathbb{C}((\xi_i)) \oplus \mathbb{C}c,$$ where $c$ belongs to the center of $\widehat{\frg}_n$ and the Lie bracket is given as follows: 
$$[\sum_{i=1}^nX_i \otimes f_i, \sum_{i=1}^nY_i\otimes g_i]:=\sum_{i=1}^n[X_i,Y_i]\otimes f_ig_i + \sum_{i=1}^n(X_i, Y_i)\Res_{\xi_i=0}(g_idf_i)c.$$ 
We define the current algebra to be $$\frg(\mathfrak{X}):=\frg\otimes \Gamma(C-\{P_1,\dots, P_n\}, \mathcal{O}_{C}).$$ By local expansion of functions using the chosen coordinates $\xi_i$, we get the following embedding:
$$\frg(\mathfrak{X}) \hookrightarrow \widehat{\frg}_{n}.$$

Consider an $n$ tuple of weights $\vec{\lambda}=(\lambda_1,\dots, \lambda_n) \in P_{\ell}^n(\frg)$. We set $\mathcal{H}_{\vec{\lambda}}=\mathcal{H}_{\lambda_1}(\frg)\otimes \dots \otimes \mathcal{H}_{\lambda_n}(\frg)$. The algebra $\widehat{\frg}_n$ acts on $\mathcal{H}_{\vec{\lambda}}$. For any $X\in \frg$ and $f\in \mathbb{C}((\xi_i))$, the action of $X\otimes f(\xi_i)$ on the $i$-th component is given by the following: 
$$\rho_i(X\otimes f(\xi_i))|v_1\otimes \dots \otimes v_n\rangle =|v_1\otimes \dots \otimes (X\otimes f(\xi_i)v_i) \otimes \dots \otimes v_n\rangle,$$ where $|v_i\rangle \in \mathcal{H}_{\lambda_i}(\frg)$ for each $i$. 
\begin{definition}
We define the space of conformal blocks 
$$\mathcal{V}^{\dagger}_{\vec{\lambda}}(\mathfrak{X}, \frg):=\operatorname{Hom}_{\mathbb{C}}(\mathcal{H}_{\vec{\lambda}}/\frg(\mathfrak{X})\mathcal{H}_{\vec{\lambda}}, \mathbb{C}).$$ 

We define the space of dual conformal blocks, $\mathcal{V}_{\vec{\lambda}}(\mathfrak{X}, \frg)=\mathcal{H}_{\vec{\lambda}}/\frg(\mathfrak{X})\mathcal{H}_{\vec{\lambda}}$. These are both finite dimensional $\mathbb{C}$-vector spaces which can be defined in families. The dimensions of these vector spaces are given by the {Verlinde formula}. 
\end{definition}
The elements of $\mathcal{V}_{\vec{\lambda}}^{\dagger}(\mathfrak{X},\frg)$ (or $\mathcal{H}_{\vec{\lambda}}^*$) will be denoted by $\langle \Psi |$ and those of the dual conformal blocks (or $\mathcal{H}_{\vec{\lambda}}$) by $|\Phi\rangle $. We will denote the natural pairing by $\langle  \Psi | \Phi \rangle$. 
\begin{remark}\label{gauge}
Let $X\in \frg$ and $f \in \Gamma(C-\{P_1,\dots, P_n\}, \mathcal{O}_C)$, then every element of $\langle \Psi|  \in \mathcal{V}_{\vec{\lambda}}^{\dagger}(\mathfrak{X}, \frg)$ satisfies the following gauge symmetry: 
$$\sum_{i=1}^n \langle \Psi |\rho_i( X\otimes f(\xi_i))\Phi \rangle =0.$$ 
\end{remark} 

\subsection{Propagation of vacua}\label{propofvacua} Let $P_{n+1}$ be a new point on the curve $C$ with coordinate $\eta_{n+1}$ and $\mathfrak{X}'$ denote the new data. We associate the vacuum representation $\mathcal{H}_{0}$ to the point $P_{n+1}$ and $\vec{\lambda}'=\vec{\lambda}\cup \{\lambda_{n+1}=0\}$. The ``propagation of vacuum" gives an isomorphism 
$$f :\mathcal{V}^{\dagger}_{\vec{\lambda}}(\mathfrak{X}', \frg)\rightarrow \mathcal{V}^{\dagger}_{\vec{\lambda}}(\mathfrak{X}, \frg)$$ by the formula 
$$f(\langle \Psi'|)|\Phi \rangle :=\langle \Psi' | \Phi\otimes 0 \rangle, $$ where $|0\rangle$ is a highest weight vector of the representation $\mathcal{H}_{0}$, $| \phi \rangle \in \mathcal{H}_{\vec{\lambda}}$ and $\langle \Psi'|$ is an arbitrary element of 
$\mathcal{V}^{\dagger}_{\vec{\lambda}}(\mathfrak{X}', \frg)$.

\subsection{Conformal blocks in a family}Let $\frg$ be a simple Lie algebra over $\mathbb{C}$ and $\vec{\lambda}\in P^n_{\ell}(\frg)$. Consider a family $\mathcal{F}=(\pi:\mathcal{C}\rightarrow \mathcal{B}; s_1,\dots, s_n; \xi_1,\dots,\xi_n)$ of nodal curves on a base $\mathcal{B}$ with sections $s_i$ and formal coordinates $\xi_i$. In ~\cite{TUY}, a locally free sheaf $\mathcal{V}^{\dagger}_{\vec{\lambda}}(\mathcal{F},\frg)$ known as the sheaf of conformal blocks is constructed over the base $\mathcal{B}$. The sheaf $\mathcal{V}^{\dagger}_{\vec{\lambda}}(\mathcal{F},\frg)$ commutes with base change. Similarly one can define another locally free sheaf $\mathcal{V}_{\vec{\lambda}}(\mathcal{F},\frg)$ as a quotient of $\mathcal{O}_{\mathcal{B}}\otimes \mathcal{H}_{\vec{\lambda}}$.

Moreover, if $\mathcal{F}$ is a family of smooth projective curves, then the sheaf $\mathcal{V}^{\dagger}_{\vec{\lambda}}(\mathcal{F},\frg)$ carries a projectively flat connection known as the $KZ/Hitchin/WZW$ connection. We refer the reader to ~\cite{TUY} for more details.

\begin{remark}
When the level $\ell$ becomes unclear, we also include it in the notation of conformal blocks. Let $\mathfrak{X}$ be the data associated to a $n$ pointed curve with chosen coordinates and $\vec{\lambda}$ be an $n$ tuple of level $\ell$ weights of the Lie algebra $\frg$. The conformal block is denoted by $\mathcal{V}^{\dagger}_{\vec{\lambda}}(\mathfrak{X},\frg,\ell)$ and the dual conformal block is denoted by $\mathcal{V}_{\vec{\lambda}}(\mathfrak{X},\frg,\ell)$.
\end{remark}

\section{Conformal Subalgebras and Rank-level duality map}\label{branching}In this section, we discuss conformal embeddings of Lie algebras and give a general formulation of rank-level duality maps. 

\subsection{Conformal embedding}
Let $\mathfrak{s}$, $\frg$ be two simple Lie algebras and $\phi : \mathfrak{s} \rightarrow \mathfrak{g}$ an embedding of Lie algebras. Let $(,)_{\mathfrak{s}}$ and $(,)_{\frg}$ denote normalized Cartan killing forms such that the length of the longest root is $2$. We define the Dynkin index of $\phi$ to be the unique integer $d_{\phi}$ satisfying 
$$(\phi(x), \phi(y))_{\frg}=d_{\phi}(x,y)_{\mathfrak{s}}$$ for all $x, y \in \mathfrak{s}$. When $\mathfrak{s}=\mathfrak{g}_1\oplus \mathfrak{g}_2$ is semisimple, we define the Dynkin multi-index of $\phi=\phi_1\oplus \phi_2:\mathfrak{g}_1\oplus\mathfrak{g}_2 \rightarrow \mathfrak{g}$ to be $d_{\phi}=(d_{\phi_1}, d_{\phi_2})$.

If $\frg$ is simple, for any  ${\lambda} \in P_{\ell}(\frg)$, we define the conformal anomaly $c(\frg, \ell)$ and the trace anomaly $\Delta_{\lambda}(\frg,\ell)$ as follows:
$$c(\frg, \ell)=\frac{\ell\dim{\frg}}{g^*+\ell} \  \ \mbox{and} \ \  \Delta_{{\lambda}}(\frg, \ell)=\frac{({\lambda}, {\lambda}+2{\rho})}{2(g^*+\ell)},$$ where $g^*$ is the dual Coxeter number of $\frg$ and ${\rho}$ denotes the half sum of positive roots, also known as the Weyl vector.
If $\frg=\frg_1 \oplus \frg_2$ is semisimple, we define the conformal anomaly and trace anomaly by taking sum of the conformal anomalies over all simple components.

\begin{definition}
Let $\phi=(\phi_1,\phi_2): \mathfrak{s}=\mathfrak{g}_1\oplus \mathfrak{g}_2 \rightarrow \frg$ be an embedding of Lie algebras with Dynkin multi-index $k=(k_1,k_2)$. We define $\phi$ to be a conformal embedding $\mathfrak{s}$ in $\mathfrak{g}$ at level $\ell$ if $c({\frg_1}, k_1\ell)+c(\frg_2,k_2\ell)=c(\mathfrak{g}, \ell)$. 
\end{definition}

It is shown in ~\cite{K} that the above equality only holds if $\ell=1$. Many familiar and important embeddings are conformal. For a complete list of conformal embeddings, we refer the reader to ~\cite{BB}. Next, we list two important properties which makes conformal embeddings special.
\begin{enumerate}
%\begin{remark}
\item Let $\mathfrak{s}=\frg_1\oplus \frg_2$ be a semisimple complex Lie algebra, then an embedding
$\phi: \mathfrak{s}\rightarrow  \frg$ is a conformal subalgebra if and only if any irreducible $\widehat{\frg}$-module
$\mathcal{H}_{\Lambda}(\frg)$ of level one decompose into a finite sum of irreducible $\widehat{\mathfrak{s}}$-modules of level $\ell=(\ell_1,\ell_2)$, where $\ell$ is the Dynkin multi-index of the embedding $\phi$. A proof of the above can be found in \cite{KW}.
%\end{remark}
%\begin{remark}
\item If $\phi: \mathfrak{s} \rightarrow \frg$ is a conformal embedding, then the action of the Virasoro operators are the same, i.e. for any integer $k$, the following equality holds: 
$$L_k^{\mathfrak{s}}=L_k^{\frg} \in \operatorname{End}(\mathcal{H}_{\Lambda}(\frg)),$$ where $L_k^{\mathfrak{s}}$ and $L_{k}^{\mathfrak{g}}$ are $k$-th Virasoro operators of $\mathfrak{s}$ and $\mathfrak{g}$ acting at level $\ell$ and one respectively. We refer the reader to \cite{KW} for more details.
%\end{remark}\label{virasoro}
\end{enumerate}

\subsection{General context of rank-level duality} Consider a level one integrable highest weight $\widehat{\frg}$-module $\mathcal{H}_{\Lambda}(\frg)$ and restrict it to $\widehat{\frg}_1\oplus \widehat{\frg}_2$. The module $\mathcal{H}_{\Lambda}(\frg)$ decomposes into irreducible integrable $\widehat{\frg}_1\oplus \widehat{\frg}_2$-modules of level $\ell=(\ell_1,\ell_2)$ as follows:

$$\bigoplus_{(\lambda,\mu)\in B(\Lambda)}m_{\lambda, \mu}^{\Lambda}\mathcal{H}_{\lambda}(\frg_1)\otimes \mathcal{H}_{\mu}(\frg_2)\simeq \mathcal{H}_{\Lambda}(\frg),$$
 where $\ell$ is the Dynkin multi-index of $\phi$ and $m_{\lambda, \mu}^{\Lambda}$ is the multiplicity of the component $\mathcal{H}_{\lambda}(\frg_1)\otimes \mathcal{H}_{\mu}(\frg_2)$. Since the embedding is conformal, we know that both $|B(\Lambda)|$ and $m_{\lambda,\mu}^{\Lambda}$ are finite.

We consider only those conformal embeddings such that for every $\Lambda \in P_1(\frg)$ and $(\lambda, \mu )\in B(\Lambda)$, the multiplicity $m_{\lambda, \mu}^{\Lambda}=1$. Let $\vec{\Lambda}=(\Lambda_1,\dots, \Lambda_n)$ be a $n$-tuple of level one dominant weights of $\frg$. We consider $\mathcal{H}_{\vec{\Lambda}}(\frg)$ and restrict it to $\widehat{\frg}_1\oplus \widehat{\frg}_2$. Choose $\vec{\lambda}=(\lambda_1,\dots,\lambda_n)$ and $\vec{\mu}=(\mu_1,\dots, \mu_n)$ such that $(\lambda_i, \mu_i) \in B(\Lambda_i)$ for all $1\leq i\leq n$. We get an injective map $$\bigotimes_{i=1}^n(\mathcal{H}_{\lambda_i}(\frg_1)\otimes \mathcal{H}_{\mu_i}(\frg_2)) \rightarrow \bigotimes_{i=1}^n\mathcal{H}_{\Lambda_i}(\frg).$$  

Let $\mathfrak{X}$ denote the data associated to a curve $C$ of genus $g$ with $n$ distinct points $\vec{p}=(P_1,\dots, P_n)$ with chosen coordinates $\xi_1, \dots, \xi_n$. Taking coinvariants with respect to $\frg(\mathfrak{X})$, we get the following map:
$$\alpha : \mathcal{V}_{\vec{\lambda}}(\mathfrak{X}, \frg_1, \ell_1 )\otimes \mathcal{V}_{\vec{\mu}}(\mathfrak{X},\frg_2, \ell_2 )\rightarrow \mathcal{V}_{\vec{\Lambda}}(\mathfrak{X}, \frg, 1),$$ 

If $\dim_{\mathbb{C}}(\mathcal{V}_{\vec{\Lambda}}(\mathfrak{X}, \frg, 1))=1$, we get a map well defined up to constants 
$$ \alpha^{\vee}: \mathcal{V}_{\vec{\lambda}}(\mathfrak{X}, \frg_1, \ell_1)\rightarrow  \mathcal{V}^{\dagger}_{\vec{\mu}}(\mathfrak{X},\frg_2, \ell_2).$$ This map is known as the rank-level duality map.

\begin{definition}\label{admiss}Let $\vec{\lambda}\in P^n_{\ell_1}(\frg_1)$ and $\vec{\mu}\in P^n_{\ell_2}(\frg_2)$. The pair $(\vec{\lambda},\vec{\mu})$ is called admissible if one can define a rank-level duality map between the corresponding conformal blocks.
\end{definition}
Let $\mathcal{F}=(\pi:\mathcal{C}\rightarrow \mathcal{B}; s_1,\dots, s_n; \xi_1,\dots, \xi_n)$ be a family of nodal curves on a base $\mathcal{B}$ with sections $s_i$ and local coordinates $\xi_i$. The map $\alpha$ can be easily extended to a map of sheaves 
$$\alpha(\mathcal{F}): \mathcal{V}_{\vec{\lambda}}(\mathcal{F}, \frg_1, \ell_1 )\otimes \mathcal{V}_{\vec{\mu}}(\mathcal{F},\frg_2, \ell_2 )\rightarrow \mathcal{V}_{\vec{\Lambda}}(\mathcal{F}, \frg, 1).$$

\subsection{Properties of rank-level duality}In this section, we recall some interesting properties of rank-level duality maps. The following proposition tells us about the behavior of the rank-level duality map in a smooth family of curves. For a proof, we refer the reader to ~\cite{Bel}.
\begin{proposition}\label{flatness}Let $\mathcal{F}=(\pi:\mathcal{C}\rightarrow \mathcal{B}; s_1,\dots, s_n; \xi_1,\dots, \xi_n)$ be a family of smooth projective curves on a base $\mathcal{B}$ with sections $s_i$ and local coordinates $\xi_i$. Then, the rank-level duality map $\alpha$ is projectively flat with respect to the $KZ/Hitchin/WZW$ connection.
\end{proposition}

The rank-level duality map commutes with the propagation of vacua. The following has a direct proof. 
\begin{proposition}\label{vacua}
Let $Q$ be a point on the curve $C$ distinct from $\vec{p}=(P_1,\dots, P_n)$ and $\mathfrak{X}'$ be the data associated to the $n+1$ pointed curve. Consider $\vec{\lambda}'=(\lambda_1,\dots, \lambda_n,0)$ and $\vec{\mu}'=(\mu_1,\dots,\mu_n,0)$. The rank-level duality map $\mathcal{V}_{\vec{\lambda}}(\mathfrak{X}, \frg_1, \ell_1)\rightarrow \mathcal{V}^{\dagger}_{\vec{\mu}}(\mathfrak{X},\frg_2, \ell_2)$ is an isomorphism if and only if the rank-level duality map 
$\mathcal{V}_{\vec{\lambda}'}(\mathfrak{X}', \frg_1, \ell_1)\rightarrow  \mathcal{V}^{\dagger}_{\vec{\mu}'}(\mathfrak{X}',\frg_2, \ell_2)$ is an isomorphism.
\end{proposition}

\subsubsection{Diagram automorphisms and rank-level duality} Let $G$ be a complex simply connected group with Lie algebra $\frg$. The center $Z(G)$ acts on the affine Lie algebra $\widehat{\frg}$ as diagram automorphisms. The action of the center $Z(G)$ preserves the Cartan subalgebra of $\widehat{\frg}$ and hence it also acts on $P_{\ell}(\frg)$. Consider the set $\Gamma(G)=\{(\sigma_1,\dots, \sigma_n)\in Z(G)^n| \prod_{i=1}^n \sigma_i=\operatorname{id}\}$. For $\vec{\sigma} \in \Gamma(G)$, we denote by $\vec{\sigma}\vec{\lambda}$, the $n$ tuple $(\sigma_1.\lambda_1,\dots,\sigma_n.\lambda_n)$, where $\sigma.\lambda$ is a level $\ell$ weight of $\frg$ given by the action of diagram automorphism $\sigma$ on a weight $\lambda \in P_{\ell}(\frg)$. The following is one of the main results in ~\cite{FS}.

\begin{proposition}\label{diaaut1}
Let $\mathfrak{X}$ be the data associated to $n$ distinct points on $\mathbb{P}^1$ with chosen coordinates. Then, there is an isomorphism 
$$\Theta_{\vec{\sigma}}(\mathfrak{X}) : \mathcal{V}_{\vec{\lambda}}(\mathfrak{X}, \frg, \ell)\rightarrow \mathcal{V}_{\vec{\sigma}\vec{\lambda}}(\mathfrak{X}, \frg, \ell).$$ More over the isomorphism is flat with respect to the KZ/Hitchin/WZW connection.
\end{proposition}

The isomorphism $\Theta_{\vec{\sigma}}(\mathfrak{X})$ in ~\cite{FS} depends on the choice of formal coordinates around the marked points and have the following functorial property under embeddings of Lie algebras. Let $G_1$, $G_2$ and $G$ be simply connected Lie groups with simple Lie algebras $\frg_1$, $\frg_2$ and $\frg$ respectively. Consider a map $\phi: G_1\times G_2 \rightarrow G$ and let $d\phi: \frg_1\oplus \frg_2 \rightarrow \frg$ be the map of Lie algebras. For any simply connected, simple Lie group $G$, consider $$\Gamma(G)=\{(\sigma_1,\dots,\sigma_n) \in Z(G)^n|\prod_{i=1}^n\sigma_i=\operatorname{id}\}.$$ We only give an outline of the proof of the following proposition. For a complete proof, see ~\cite{M1}. We assume that the dimension of $\mathcal{V}_{\vec{\Lambda}}(\mathfrak{X},\frg,1)$ is one. 
%\section{Flatness of Rank level duality map}
\begin{proposition}\label{diaaut2} Let $\vec\Sigma \in \Gamma (G)$, $\vec{\sigma}\in \Gamma(G_1)$ be such that $\phi(\vec{\sigma})=\vec{\Sigma}$, then the pairing $$\mathcal{V}_{\vec{\lambda}}(\mathfrak{X}, \frg_1, \ell_1)\otimes  \mathcal{V}_{\vec{\mu}}(\mathfrak{X},\frg_2, \ell_2)\rightarrow \mathcal{V}_{\vec{\Lambda}}(\mathfrak{X},\frg, 1)$$ is non-degenerate if and only if the following pairing is non-degenerate: $$\mathcal{V}_{\vec{\sigma} \vec{\lambda}}(\mathfrak{X}, \frg_1, \ell_1)\otimes \mathcal{V}_{\vec{\mu}}(\mathfrak{X},\frg_2, \ell_2) \rightarrow \mathcal{V}_{\vec{\Sigma}\vec{\Lambda}}(\mathfrak{X},\frg, 1).$$ 
\end{proposition}

\begin{proof}
For every $n$ tuple $\vec{\sigma}=(\sigma_1,\dots,\sigma_i,\dots, \sigma_n)$ (respectively $\vec{\Sigma}$) as in the proposition, $n$ automorphisms of $\widehat{\frg}_1$ (respectively $\widehat{\frg}$) are constructed in \cite{FS}. These automorphisms are known as multi-shift automorphisms. We refer the reader to \cite{FS, M1} for a definition. Multishift automorphisms for $\vec{\sigma}$ (respectively $\vec{\Sigma}$) have the following features:
\begin{itemize}
\item If $\sigma_i$ is non-trivial, then the associated multi-shift automorphism is not an inner automorphism.
\item For $\vec{\sigma}$ (respectively $\vec{\Sigma}$), it gives an automorphism of $\bigoplus_{i=1}^n\frg_1\otimes \mathbb{C}((\xi_i))\oplus \mathbb{C}c$ (respectively for $\bigoplus_{i=1}^n\frg\otimes \mathbb{C}((\xi_i))\oplus \mathbb{C}c$ ).
\item The above automorphism preserves the subalgebra $\frg_1(\mathfrak{X})$ (respectively $\frg(\mathfrak{X})$).
\end{itemize}
In \cite{M1}, with the assumption of Proposition \ref{diaaut2}, it is shown that the multishift automorphisms for $\vec{\sigma}$ commute with multishift automorphisms for $\vec{\Sigma}$ for the map $\widehat{\frg}_1\oplus \widehat{\frg}_2 \rightarrow \widehat{\frg}$. This and the proof of Proposition \ref{diaaut1} (see \cite{FS}) give us the following diagram of conformal blocks which commutes:
$$ 
\xymatrix{
\mathcal{V}_{\vec{\lambda}}(\mathfrak{X}, \frg_1, \ell_1)\otimes  \mathcal{V}_{\vec{\mu}}(\mathfrak{X},\frg_2, \ell_2) \ar[d]^{\Theta_{\vec{\sigma}}(\mathfrak{X})\otimes \operatorname{id}} \ar[r] & \mathcal{V}_{\vec{\Lambda}}(\mathfrak{X},\frg, 1)\ar[d]^{\Theta_{\vec{\Sigma}}(\mathfrak{X})}\\
\mathcal{V}_{\vec{\sigma}\vec{\lambda}}(\mathfrak{X}, \frg_1, \ell_1)\otimes  \mathcal{V}_{\vec{\mu}}(\mathfrak{X},\frg_2, \ell_2)\ar[r] & \mathcal{V}_{\vec{\Sigma}\vec{\Lambda}}(\mathfrak{X},\frg, 1)}
$$ Here the vertical arrows are the isomorphisms constructed in \cite{FS} and the horizontal arrows are rank-level duality maps that come from the conformal embedding $\frg_1\oplus \frg_2\rightarrow \frg$. The proof of the proposition now follows easily.
\end{proof}

\section{Sewing and compatibility under factorization} In this section, we recall the sewing construction from ~\cite{TUY}. We consider a family of curves degenerating to a curve with one node. We study compatibility of rank-level duality map with factorization of a nodal curve following ~\cite{BP}. We will use Proposition \ref{keydegen} to reduce rank-level duality questions on $n$ pointed curves to rank-level duality for certain one dimensional conformal blocks on $\mathbb{P}^1$ with three marked points. Our strategy is inspired by Proposition 5.2 in ~\cite{Pau}. We refer the reader to Section 9 for more details. First we recall the following lemma from ~\cite{TUY}. 
\begin{lemma}\label{funlemma}
There exists a bilinear pairing 
$$(, )_{\lambda}: \mathcal{H}_{\lambda}\times \mathcal{H}_{\lambda^{\dagger}} \rightarrow \mathbb{C} $$ unique up to a multiplicative constant such that 
 $$(X(n)u, v )_{\lambda} + (u, X(-n)v)_{\lambda}=0.$$ for all $X\in \frg$, $n \in \mathbb{Z}$, $u\in \mathcal{H}_{\lambda}$ and $v \in \mathcal{H}_{\lambda^{\dagger}}$. Moreover, the restriction of the form $(,)_{\lambda}$ to $\mathcal{H}_{\lambda}(m)\times \mathcal{H}_{\lambda^{\dagger}}(m')$ is zero if $m\neq m'$ and is non-degenerate if $m=m'$.
\end{lemma}

%Next we show that we can realize the bilinear form $(,)_\lambda$ as an element of a suitable conformal block. 
%\begin{lemma} 
%For $\lambda \in P_k(\frg)$ consider the conformal block $\mathcal{V}_{\vec{\lambda}}^{\dagger}(\mathfrak{X})$, where $\mathfrak{X}$ is the data associated to the projective line with two distinct marked points and $\vec{\lambda}=(\lambda, \lambda^{\dagger})$. Then, $(,)_{\lambda}\in \mathcal{V}_{\vec{\lambda}}^{\dagger}(\mathfrak{X}).$

%\end{lemma}

%\begin{proof}
%Let $X[f] \in \frg(\mathfrak{X})$. We need to show: $$(\rho_1(X[f])u, v)_{\lambda} + (u, \rho_2(X[f])v)_{\lambda}=0,$$ for any $u \in \mathcal{H}_{\lambda}$ and $v \in \mathcal{H}_{\lambda^{\dagger}}$.

%\end{proof}
%Let $\mathfrak{p} \rightarrow \frg$ be a conformal embedding. The restriction of the integrable highest weight modules $\mathcal{H}_{\lambda}(\frg)$ and $\mathcal{H}_{\lambda^{\dagger}}$ to $\widehat{\mathfrak{p}}$ decomposes as follows:
%$$\mathcal{H}_{\lambda}(\frg)=\bigoplus_{\widehat{\mu} \in B(\lambda)} \mathcal{H}_{\widehat{\mu}}(\mathfrak{p}),
 %\ \ \mathcal{H}_{\lambda^{\dagger}}(\frg)=\bigoplus_{\widehat{\mu} \in B(\lambda)} \mathcal{H}_{\widehat{\mu}^{\dagger}}(\mathfrak{p}).$$ 
 
 Since the restriction of the bilinear form $(,)_{\lambda}$ to $\mathcal{H}_{\lambda}(m)\times\mathcal{H}_{\lambda^{\dagger}}(m)$ is non-degenerate, we obtain an isomorphism of $\mathcal{H}_{\lambda^{\dagger}}(m)$ with $\mathcal{H}_{\lambda}(m)^*$. Let $\gamma_{\lambda}(m)$ be the distinguished element of $\mathcal{H}_{\lambda}(m)\otimes \mathcal{H}_{\lambda^{\dagger}}(m)$ given by $(,)_{\lambda}$. Let $t$ be a formal variable. Given $\lambda \in P_{\ell}(\frg)$, we construct an element $\widetilde{\gamma}_{\lambda}=\sum_{m=0}^{\infty}\gamma_{\lambda}(m)t^m$ of $\mathcal{H}_{\lambda}\otimes \mathcal{H}_{\lambda^{\dagger}}[[t]]$.

We are now ready to describe the sewing procedure in ~\cite{TUY}. Throughout the section, let $\mathcal{B}=\operatorname{Spec}\mathbb{C}[[t]]$. We consider a family of curves $\mathcal{X} \rightarrow \mathcal{B}$ with $n$ marked points with chosen coordinates such that it's special fiber $\mathcal{X}_0$ is a curve $X_0$ over $\mathbb{C}$ with exactly one node and it's generic fiber $\mathcal{X}_t$ is a smooth curve. Consider the sheaf of conformal blocks $\mathcal{V}^{\dagger}_{\vec{\lambda}}(\mathcal{X}, \frg)$ for the family of curves $\mathcal{X}$. The sheaf of conformal blocks commutes with base change and the fiber over any point $t \in \mathcal{B}$ coincides with $\mathcal{V}_{\vec{\lambda}}^{\dagger}(\mathfrak{X}_t, \frg)$, where $\mathfrak{X}_t$ is the data associated to the curve $X_t$ over the point $t\in \mathcal{B}$.

Let $\widetilde{X}_0$ be the normalization of $X_0$. For $\lambda \in P_{\ell}(\frg)$, the following isomorphism is constructed in ~\cite{TUY}
$$\oplus \iota_{\lambda}: \bigoplus_{\lambda \in P_{\ell}(\frg)} \mathcal{V}_{\lambda,\lambda^{\dagger}, \vec{\lambda}}^{\dagger}(\widetilde{\mathfrak{X}}, \frg)\rightarrow \mathcal{V}_{\vec{\lambda}}^{\dagger}(\mathfrak{X}, \frg),$$ where $\widetilde{\mathfrak{X}}$ is the data associated to the $(n+2)$ points of the smooth pointed curve $\widetilde{X}_0$ with chosen coordinates. 

In ~\cite{TUY}, a sheaf version of the above isomorphism is also proved. We briefly recall the construction. For every $\lambda \in P_{\ell}(\frg)$ there exists a map 
$$s_{\lambda} :\mathcal{V}_{\lambda,\lambda^{\dagger}, \vec{\lambda}}^{\dagger}(\widetilde{\mathfrak{X}}, \frg)\rightarrow \mathcal{V}_{\vec{\lambda}}^{\dagger}(\mathcal{X}, \frg),$$ where $s_{\lambda}(\psi)=\widetilde{\psi}$ and $\widetilde{\psi}(\widetilde{u}):=\psi(\widetilde{u}\otimes \widetilde{\gamma}_{\lambda})\in \mathbb{C}[[t]]$ for any $\widetilde{u}\in \mathcal{H}_{\vec{\lambda}}[[t]]$. This map extends to a map $s_{\lambda}(t)$ of coherent sheaves of $\mathbb{C}[[t]]$-modules
$$s_{\lambda}(t) : \mathcal{V}_{\lambda,\lambda^{\dagger}, \vec{\lambda}}^{\dagger}(\widetilde{\mathfrak{X}}, \frg)\otimes \mathbb{C}[[t]]\rightarrow \mathcal{V}_{\vec{\lambda}}^{\dagger}(\mathcal{X}, \frg).$$
 With the above notation, the following is proved in ~\cite{TUY}. We also refer the reader to Theorem 6.1 in ~\cite{Lo1}.
\begin{proposition}
The map $$\oplus s_{\lambda}(t): \bigoplus_{\lambda \in P_{\ell}(\frg)} \mathcal{V}_{\lambda,\lambda^{\dagger}, \vec{\lambda}}^{\dagger}(\widetilde{\mathfrak{X}}, \frg)\otimes \mathbb{C}[[t]]\rightarrow \mathcal{V}_{\vec{\lambda}}^{\dagger}(\mathcal{X}, \frg).$$ is an isomorphism of locally free sheaves on $\mathcal{B}$. 
\end{proposition} 
%We refer the reader to ~\cite{TUY} for a proof.
 \subsection{Factorization and compatibility of rank-level duality}
 Consider a conformal embedding $\mathfrak{s}\rightarrow \frg$. Assume that all level one highest weight integrable modules of $\widehat{\frg}$ decompose with multiplicity one as $\widehat{\mathfrak{s}}$-modules.
 
 Let $\vec{\lambda}=(\lambda_1, \dots, \lambda_n)$ be an $n$ tuple of level one weights of $\frg$ and $\vec{{\mu}}\in B(\vec{\lambda})$. We get a map $\mathcal{H}_{\vec{{\mu}}}(\mathfrak{s}) \rightarrow \mathcal{H}_{\vec{\lambda}}(\mathfrak{g}).$ As discussed in Section 3, we also get a $\mathbb{C}[[t]]$-linear map 
 $$\alpha(t) :\mathcal{V}_{\vec{\lambda}}^{\dagger}(\mathcal{X}, \frg) \rightarrow \mathcal{V}_{\vec{\mu}}^{\dagger}(\mathcal{X}, \mathfrak{s}).$$ For ${\mu} \in B(\lambda)$, we denote by $\alpha_{\lambda, \mu}$ the map induced from branching as discussed in Section 3 
 $$\mathcal{V}^{\dagger}_{\lambda, \lambda^{\dagger},\vec{\lambda}}(\widetilde{X}_0, \frg) \rightarrow \mathcal{V}^{\dagger}_{\mu, \mu^{\dagger}, \vec{\mu}}(\widetilde{X}_0, \mathfrak{s})$$ and the extension of $\alpha_{\lambda, \mu}$ to a $\mathbb{C}[[t]]$-linear map is denoted as follows:
 $$\alpha_{\lambda, \mu}(t): \mathcal{V}^{\dagger}_{\lambda, \lambda^{\dagger},\vec{\lambda}}(\widetilde{X}_0, \frg) \otimes \mathbb{C}[[t]]\rightarrow \mathcal{V}^{\dagger}_{\mu, \mu^{\dagger}, \vec{\mu}}(\widetilde{X}_0, \mathfrak{s})\otimes \mathbb{C}[[t]].$$
 
The following proposition from \cite{BP} describes how $\alpha(t)$ decomposes under factorization. 
\begin{proposition}\label{keydegen}On $\mathcal{B}$, we have 

$$\alpha(t) \circ s_{\lambda}(t)=\sum_{{\mu} \in B(\lambda)}t^{n_{\mu}}s_{\mu}(t)\circ\alpha_{\lambda,\mu}(t),$$ where $n_{\mu}$ are positive integers given by the formula:
 $$ n_{\mu}=\Delta_{\mu}-\Delta_{\lambda}.$$
\end{proposition}

\begin{remark}
The integers $n_{\mu}$ are non-zero if the finite dimensional $\mathfrak{s}$-module $V_{\mu}$ does not appear in the decomposition of the finite dimensional $\frg$-module $V_{\lambda}$.% These integers $n_{\mu}$ controls the nature of the rank-level duality map on a nodal curve.
\end{remark}

\section{Branching Rules for conformal embedding of orthogonal Lie algebras}\label{branching2}In this section, we discuss the branching rule for the conformal embedding $\mathfrak{so}(2r+1)\oplus \mathfrak{so}(2s+1)\rightarrow \mathfrak{so}((2r+1)(2s+1))$.

\subsection{Representation of $\mathfrak{so}(2r+1)$}Let $E_{i,j}$ be a matrix whose $(i,j)$-th entry is one and all other entries are zero. The Cartan subalgebra $\mathfrak{h}$ of $\mathfrak{so}(2r+1)$ is generated by diagonal matrices of the form $H_i=E_{i,i}-E_{r+i, r+i}$ for $1\leq i \leq r$. Let $L_i \in \mathfrak{h}^*$ be defined by $L_i(H_j)=\delta_{i,j}$. The normalized Cartan killing form on $\mathfrak{h}$ is given by $(H_i, H_j)=\delta_{ij}$. Under the identification of $\mathfrak{h}^*$ with $\mathfrak{h}$ using the Cartan Killing form, the image of $L_i$ is $H_i$ for all $1\leq i \leq r$.
%We follow the notation of ~\cite{FH} for describing the root systems and weight lattice. We identify $B_n$ as a subset of the set of $(2n+1)\times(2n+1)$ matrices with complex matrices.

We can choose the simple positive roots of $\mathfrak{so}(2r+1)$ to be $\alpha_1=L_1-L_2, \alpha_2=L_2-L_3, \dots, \alpha_{r-1}=L_{r-1}-L_r, \alpha_r= L_r$. The highest root is $\theta=L_1+ L_2 =\alpha_1 + 2\alpha_2 + \dots + 2\alpha_r$. The fundamental weights of $B_r$ are $\omega_i=L_1 + L_2 +\dots + L_i$ for $1\leq i < r$ and $\omega_r=\frac{1}{2}(L_1 + L_2 +\dots + L_r)$. 

The dominant integral weights, $P_{+}$, of $\mathfrak{so}(2r+1)$ can be written as 
$$P_{+}= P_{+}^0 \sqcup P_{+}^1,$$ where $P^0_{+}$ is the set of dominant weights $\lambda=\sum_{i=1}^r a_i\omega_i$ such that $a_r$ is even and $P_{+}^1:=P_{+}^0 + \omega_r$. Let $\mathcal{Y}_r$ be the set of Young diagrams with at most $r$ rows and $\mathcal{Y}_{r,s}$ denote the set of Young diagrams with at most $r$ rows and $s$ columns. Then, the set $P_{+}^0$ is in bijection with $\mathcal{Y}_r$.

 Combinatorially any dominant weight $\lambda$ of $P_{+}$ can be written as $Y+ t\omega_r$, where $t=\{0,1\}$ and $Y\in \mathcal{Y}_r$. If $t=0$, then $\lambda \in P_{+}^0$ and if $t=1$, then $\lambda \in P_{+}^1$.

Let $\lambda=\sum_{i=1}^ra_i\omega_i$ be a dominant integral weight. Then, $$(\theta , \lambda)=a_1 + 2(a_2 + \dots + a_{r-1})+ a_r.$$ The set of level $2s+1$ dominant weights are described below: $$P_{2s+1}(\mathfrak{so}(2r+1))=\{\lambda \in P_{+} | a_1 + 2(a_2 + \dots + a_{r-1})+ a_r \leq 2s+1\}.$$ 

\subsection{The action of center on weights}
An element $\sigma$ of the center of the group $\operatorname{Spin}(2r+1)$ acts as an outer automorphism on affine Lie algebra $\widehat{\mathfrak{so}}(2r+1)$. For details, we refer the reader to ~\cite{H}. The action of $\sigma$ on $P_{2s+1}(\mathfrak{so}(2r+1))$ is given by $\sigma(\lambda)=(2s+1-(a_1 +2(a_2+\dots + a_{r-1})+a_r))\omega_1 + a_2\omega_2 + \dots + a_r\omega_r$. 
We denote the intersection $P_{+}^0\cap P_{2s+1}(\mathfrak{so}(2r+1))$ by $P_{2s+1}^0(\mathfrak{so}(2r+1))$. The following lemma can be proved by direct calculation:
\begin{lemma}
The action of $\sigma$ preserves the set $P_{2s+1}^0(\mathfrak{so}(2r+1))$. Moreover, we get $P_{2s+1}^0(\mathfrak{so}(2r+1))=\mathcal{Y}_{r,s}\sqcup \sigma(\mathcal{Y}_{r,s})$.
\end{lemma}

Following ~\cite{OW}, we describe the orbits of $P_{2s+1}(\mathfrak{so}(2r+1))$ under the action of the center. Let $\rho=\sum_{i=1}^r\omega_i$ be the Weyl vector. For $\lambda=\sum_{i=1}^r a_i\omega_i$, the weight $\lambda+\rho=\sum_{i=1}^rt_i\omega_i$, where $t_i=a_i+1$. Put $u_i=\sum_{j=i}^{r-1}t_j + \frac{t_r}{2}$ for $1\leq i\leq r$, $u_r=\frac{t_r}{2}$ and $u_{r+1}=0$.

The set $P_{2s+1}(\mathfrak{so}(2r+1))$ is identified with the collection of sets $U=(u_1> u_2 > \dots >u_r >0)$ such that 
\begin{itemize}
\item $u_i \in \frac{1}{2}\mathbb{Z}$,
\item $u_i-u_{i+1}\in \mathbb{Z}$,
\item $u_1+u_2 \leq 2(r+s)$.

\end{itemize}

Let $P^{+}_{2s+1}(\mathfrak{so}(2r+1))$ denote the set of weights in $P_{2s+1}(\mathfrak{so}(2r+1))$ such that $u_i \in \mathbb{Z}$.

Let us set $k=2(r+s)$, and we rewrite the action of the center $\Gamma$ on $P_{2s+1}(\mathfrak{so}(2r+1))$ as exchanging $t_1$ with $t_0=k-t_1-2t_2-\dots -2t_{r-1}-t_r$; or in other words changing $u_1$ with $k-u_1$. We observe that the action of $\Gamma$ preserves $P^+_{2s+1}(\mathfrak{so}(2r+1))$ and $P^0_{2s+1}(\mathfrak{so}(2r+1))$. Then, we can identify the orbits of the action of $\Gamma$ as follows :
$$P_{2s+1}(\mathfrak{so}(2r+1))/\Gamma= \{ U=(u_1,\dots, u_r)| \frac{k}{2}\geq u_1>\dots >u_r>0, u_i\in \frac{1}{2}\mathbb{Z}, u_i-u_{i+1}\in \mathbb{Z}\}$$ and the length of the orbits are given as follows:
\begin{itemize}
\item $|\Gamma(U)|=2$ if $u_1 <\frac{k}{2}$.

\item $|\Gamma(U)|=1$ if $u_1=\frac{k}{2}$.

\end{itemize}

For any number $a$ and a set $U=(u_1>u_2>\dots>u_r)$, we denote by $U-a$ and $a-U$, the sets $\{u_1-a>u_2-a>\dots>u_r-a\}$ and $\{a-u_r>a-u_{r-1}>\dots>a-u_1\}$ respectively. Further, the set $\{1,2,\dots, r+s\}$ is denoted by $[r+s]$. The following two lemmas from ~\cite{OW} give a bijection of orbits. We will use this in Section 7 to show that the source and the target of the rank-level duality maps have the same dimension.

\begin{lemma}\label{orbitbijection}Let $P_{2r+1}(\mathfrak{so}(2s+1))$ denote the weights of $\mathfrak{so}(2s+1)$ of level $2r+1$. Then, there is a bijection between the orbits of $P_{2s+1}^{+}(\mathfrak{so}(2s+1))$ and the orbits of $ P_{2r+1}^{+}(\mathfrak{so}(2s+1))$ given by $$U=(u_1> u_2> \dots >u_r) \rightarrow U^c=(u^c_1>\dots>u^c_s),$$ where $U\subset [r+s]$ of cardinality $r$ and $U^c$ is the complement of $U$ in $[r+s]$.
%$$ P_{2s+1}^{0}(\mathfrak{so}(2s+1))\rightarrow P_{2r+1}^{0}(\mathfrak{so}(2s+1))$$

\end{lemma}
For $\lambda \in P^0_{2s+1}(\mathfrak{so}(2r+1))$, we write $\lambda + \rho =\sum_{i=1}^r (u_i'-\frac{1}{2})L_i$, where $u_i'$ are all integers. We identify the identify the orbits of $P^0_{2s+1}(\mathfrak{so}(2r+1))$ under $\Gamma$ as subsets $U'=(u_1'>u_2'>\dots > u_r')$ of $[r+s]$.
\begin{lemma}

There is a bijection between the orbits of $P_{2s+1}^{0}(\mathfrak{so}(2s+1))$ and the orbits of $ P_{2r+1}^{0}(\mathfrak{so}(2s+1))$ given by $$U'=(u'_1> u'_2> \dots >u'_r) \rightarrow((r+s)+ 1- U'^c)=(u''_1>\dots>u''_s),$$ where $U'\subset [r+s]$ of cardinality $r$ and $U'^c$ is the complement of $U'$ in $[r+s]$.

\end{lemma}

\subsection{Branching rules}\label{branching1}
We now describe the branching rules for the conformal embedding $\mathfrak{so}(2r+1) \oplus \mathfrak{so}(2s+1)\subset \mathfrak{so}((2r+1)(2s+1))$. Let $N=(2r+1)(2s+1)=2d+1$. The level one highest weights of $\widehat{\mathfrak{so}}(N)$ are $0$, $\omega_1$ and $\omega_d$. The following proposition gives the decomposition of level one integrable highest weight modules of weight $0$ and $\omega_1$. We refer the reader to ~\cite{H} for a proof.
\begin{proposition}\label{branch1}
Let $\mathcal{H}_{0}(\mathfrak{so}(N))$ and $\mathcal{H}_{1}(\mathfrak{so}(N))$ denote the highest weight integrable modules of the affine Lie algebra $\widehat{\mathfrak{so}}(2d+1)$ with highest weight $0$ and $\omega_1$ respectively. Then, the module $\mathcal{H}:=\mathcal{H}_{0}(\mathfrak{so}(N))\oplus \mathcal{H}_{1}(\mathfrak{so}(N))$ breaks up as a direct sum of highest weight integrable modules of $\widehat{\mathfrak{so}}(2r+1)\oplus \widehat{\mathfrak{so}}(2s+1)$ of the form:
\begin{itemize}
\item $\mathcal{H}_{\lambda}(\mathfrak{so}(2r+1))\otimes \mathcal{H}_{\lambda^T}(\mathfrak{so}(2s+1))$,

\item $\mathcal{H}_{\lambda}(\mathfrak{so}(2r+1))\otimes \mathcal{H}_{\sigma\lambda^T}(\mathfrak{so}(2s+1))$,

\item $\mathcal{H}_{\sigma \lambda}(\mathfrak{so}(2r+1))\otimes \mathcal{H}_{\lambda^T}(\mathfrak{so}(2s+1))$,

\item $\mathcal{H}_{\sigma \lambda}(\mathfrak{so}(2r+1))\otimes \mathcal{H}_{\sigma\lambda^T}(\mathfrak{so}(2s+1))$,

\end{itemize}
where $\lambda \in \mathcal{Y}_{r,s}$ and $\sigma$ is an automorphism associated to the center of $\operatorname{Spin}(2r+1)$. Moreover, all of the above factors appear with multiplicity one.
%\begin{eqnarray*}
%\mathcal{H}= \bigoplus_{\lambda \in \mathcal{Y}_{r,s}} \mathcal{H}_{\lambda}(\mathfrak{so}(2r+1))\otimes \mathcal{H}_{\lambda^T}(\mathfrak{so}(2s+1))  \oplus \mathcal{H}_{\sigma(\lambda)}(\mathfrak{so}(2r+1))\otimes \mathcal{H}_{\sigma(\lambda^T)}(\mathfrak{so}(2s+1))
%\end{eqnarray*}

\end{proposition}
We need to determine which factor in the above decomposition rules comes from $\mathcal{H}_{0}(\mathfrak{so}(N))$ and which factor comes from $\mathcal{H}_{1}(\mathfrak{so}(N))$. The following lemma gives the trace anomaly of the level one weights $0$ and $\omega_1$ of $\widehat{\mathfrak{so}}(2d+1)$. 
\begin{lemma}
$$\Delta_{0}(\mathfrak{so}(N))=0 \ \ \Delta_{\omega_1}(\mathfrak{so}(N))=\frac{1}{2}.$$
\end{lemma}

In order to determine the components, we need to know the trace anomalies for the weight $(\lambda, \lambda^T)$. 
\begin{lemma}
For $\lambda \in \mathcal{Y}_{r,s}$, we have the following equality 
$$\Delta_{\lambda}(\mathfrak{so}(2r+1)) + \Delta_{\lambda^T}(\mathfrak{so}(2s+1))=\frac{1}{2}|\lambda|.$$
\end{lemma}
\begin{corollary}
Let $\mathcal{H}_{0}(\mathfrak{so}(N))$ denote the level one highest weight integrable $\widehat{\mathfrak{so}}(N)$-module of weight $0$ and $\lambda \in \mathcal{Y}_{r,s}$. Then, the following factors appear in the decomposition of $\mathcal{H}_{0}(\mathfrak{so}(N))$ as $\widehat{\mathfrak{so}}(2r+1)\oplus\widehat{ \mathfrak{so}}(2s+1)$-modules.

\begin{itemize}
\item $\mathcal{H}_{\lambda}(\mathfrak{so}(2r+1))\otimes \mathcal{H}_{\lambda^T}(\mathfrak{so}(2s+1))$, when $|\lambda|\text{ is even}$.
\item $\mathcal{H}_{\sigma\lambda}(\mathfrak{so}(2r+1))\otimes \mathcal{H}_{\sigma \lambda^T}(\mathfrak{so}(2s+1))$, when $|\lambda| \text{ is even}$.
\item $\mathcal{H}_{\sigma \lambda}(\mathfrak{so}(2r+1))\otimes \mathcal{H}_{\lambda^T}(\mathfrak{so}(2s+1))$, when $|\lambda| \text{ is odd}$. 
\item $\mathcal{H}_{\lambda}(\mathfrak{so}(2r+1))\otimes \mathcal{H}_{\sigma \lambda^T}(\mathfrak{so}(2s+1))$, when $|\lambda|\text{ is odd}$.
\end{itemize}
\end{corollary}

\begin{corollary}
Let $\mathcal{H}_{1}(\mathfrak{so}(N))$ denote the level one highest weight integrable $\widehat{\mathfrak{so}}(N)$-module of weight $\omega_1$ and $\lambda \in \mathcal{Y}_{r,s}$. Then, the following factors appears in the decomposition of $\mathcal{H}_{0}(\mathfrak{so}(N))$ as $\widehat{\mathfrak{so}}(2r+1)\oplus\widehat{ \mathfrak{so}}(2s+1)$-modules.

\begin{itemize}
\item $\mathcal{H}_{\lambda}(\mathfrak{so}(2r+1))\otimes \mathcal{H}_{\lambda^T}(\mathfrak{so}(2s+1))$, when $|\lambda|\text{ is odd}$.
\item $\mathcal{H}_{\sigma\lambda}(\mathfrak{so}(2r+1))\otimes \mathcal{H}_{\sigma \lambda^T}(\mathfrak{so}(2s+1))$, when $|\lambda| \text{ is odd}$.
\item $\mathcal{H}_{\sigma \lambda}(\mathfrak{so}(2r+1))\otimes \mathcal{H}_{\lambda^T}(\mathfrak{so}(2s+1))$, when $|\lambda|\text{ is even}$. 
\item $\mathcal{H}_{\lambda}(\mathfrak{so}(2r+1))\otimes \mathcal{H}_{\sigma \lambda^T}(\mathfrak{so}(2s+1))$, when $|\lambda| \text{ is even}$.
\end{itemize}
\end{corollary}
\begin{remark}
 The only components that appear in the decomposition of standard and trivial representations of the finite dimensional Lie algebra $\mathfrak{so}(2d+1)$ into $\mathfrak{so}(2r+1)\oplus\mathfrak{so}(2s+1)$-modules are $\lambda=\omega_1$ and $\lambda=0$. Hence most of the components described in the branching rules described above do not come from finite dimensional branching. This is a major difference in the branching law considered by T. Abe in \cite{A}. 

Further, this is also the main obstruction to construct rank-level duality maps in Theorem \ref{main} geometrically. It is important to study this map geometrically to understand rank-level duality on curves of higher genus. This will be considered in a subsequent project. 
\end{remark}

\section{Rank-level duality map.}In this section, we describe rank-level duality maps using the branching rule. We consider the following weights: 
\begin{itemize}
\item  $\vec{\lambda}_i=(\lambda_{i_1}, \lambda_{i_2}, \dots, \lambda_{i_{n_1}})$ and $\vec{\lambda}_{i}^T=(\lambda_{i_1}^T, \lambda_{i_2}^T, \dots, \lambda_{i_{n_1}}^T)$, where $\lambda_{i_a}\in \mathcal{Y}_{r,s}$ such that $|\lambda_{i_a}|$ is odd for each $1\leq a \leq n_1$. 

\item $\vec{\lambda}_j=(\sigma (\lambda_{j_1}), \dots, \sigma (\lambda_{j_{n_2}}))$ and $\vec{\lambda}_j^T=(\sigma (\lambda_{j_1}^T), \dots, \sigma (\lambda_{j_{n_2}}^T))$, where $\lambda_{j_a}\in \mathcal{Y}_{r,s}$ such that $|\lambda_{j_a}|$ is odd for all $1\leq a \leq n_2$

\item $\vec{\lambda}_k=(\sigma (\lambda_{k_1}), \dots, \sigma (\lambda_{k_{n_3}}))$ and $\vec{\lambda}_k^T=(\lambda_{k_1}^T, \dots,  \lambda_{k_{n_3}}^T)$, where $\lambda_{j_a}\in \mathcal{Y}_{r,s}$ such that $|\lambda_{k_a}|$ is even for all $1\leq a \leq n_3$.

\item $\vec{\lambda}_l=(\lambda_{l_1}, \lambda_{l_2}, \dots, \lambda_{l_{n_4}})$ and $\vec{\lambda}_{l}^T=(\sigma(\lambda_{l_1}^T), \dots, \sigma(\lambda_{l_{n_4}}^T))$, where $\lambda_{l_a}\in \mathcal{Y}_{r,s}$ such that $|\lambda_{l_a}|$ is even for each $1\leq a \leq n_4$. 

\item  $\vec{\beta}_i=(\beta_{i_1},\dots, \beta_{i_{m_1}})$ and $\vec{\beta}_{i}^T=(\beta_{i_1}^T, \beta_{i_2}^T, \dots, \beta_{i_{m_1}}^T)$, where $\lambda_{i_a}\in \mathcal{Y}_{r,s}$ such that $|\lambda_{i_a}|$ is even for each $1\leq a \leq m_1$. 

\item $\vec{\beta}_j=(\sigma (\beta_{j_1}), \dots, \sigma (\beta_{j_{m_2}}))$ and $\vec{\beta}_j^T=(\sigma (\beta_{j_1}^T), \dots, \sigma (\beta_{j_{m_2}}^T))$, where $\beta_{j_a}\in \mathcal{Y}_{r,s}$ such that $|\beta_{j_a}|$ is even for all $1\leq a \leq m_2$

\item $\vec{\beta}_k=(\sigma (\beta_{k_1}), \dots, \sigma (\beta_{k_{m_3}}))$ and $\vec{\beta}_k^T=(\beta_{k_1}^T, \dots,  \beta_{k_{m_3}}^T)$, where $\lambda_{j_a}\in \mathcal{Y}_{r,s}$ such that $|\beta_{k_a}|$ is odd for all $1\leq a \leq m_3$.

\item $\vec{\beta}_l=(\beta_{l_1}, \beta_{l_2}, \dots, \beta_{l_{m_4}})$ and $\vec{\beta}_{l}^T=(\sigma(\beta_{l_1}^T),  \dots,\sigma( \beta_{l_{m_4}}^T))$, where $\beta_{l_a}\in \mathcal{Y}_{r,s}$ such that $| \beta_{l_a}|$ is odd for each $1\leq a \leq m_4$. 
\end{itemize}

Let $n=\sum_{i=1}^4(n_i+m_i)$ be a positive integer, $\vec{\lambda}=\vec{\lambda}_i\cup\vec{\lambda}_j\cup \vec{\lambda}_k\cup\vec{\lambda}_l$, $\vec{\lambda}^T=\vec{\lambda}_i^T\cup\vec{\lambda}_j^T\cup \vec{\lambda}_k^T\cup\vec{\lambda}_l^T$, $\vec{\beta}=\vec{\beta}_i\cup\vec{\beta}_j\cup \vec{\beta}_k\cup\vec{\beta}_l$, $\vec{\beta}^T=\vec{\beta}_i^T\cup\vec{\beta}_j\cup \vec{\beta}_k^T\cup\vec{\beta}_l^T$ and $\mathfrak{X}$ be the data associated to $n$ distinct points on $\mathbb{P}^1$ with chosen coordinates. Then, we have the following map between conformal blocks:
$$
\alpha: \mathcal{V}_{\vec{\lambda}\cup \vec{\beta}}(\mathfrak{X}, \mathfrak{so}(2r+1),2s+1)\otimes \mathcal{V}_{\vec{\lambda}^T\cup \vec{\beta}^T}(\mathfrak{X}, \mathfrak{so}(2s+1),2r+1) \rightarrow \mathcal{V}_{\vec{\omega}_1\cup \vec{0}}(\mathfrak{X}, \mathfrak{so}(N),1),$$ where $\vec{\omega}_1=(\omega_1,\dots ,\omega_1)$ is an $(n_1+n_2+n_3+n_4)$ tuple of $\omega_1$'s and $\vec{0}=(0,\dots, 0)$ be an $(m_1+m_2+m_3+m_4)$ tuple of $0$'s. 

Assume that $(n_1+n_2+n_3+n_4)$ is even, then $\dim_{\mathbb{C}}\mathcal{V}_{\vec{\omega}_1\cup \vec{0}}(\mathfrak{X}, \mathfrak{so}(N),1)=1$ (see Section 7.1.1). Thus, we have the following map:
\begin{equation}\label{rld}
 \alpha^{\vee}:\mathcal{V}_{\vec{\lambda}\cup \vec{\beta}}(\mathfrak{X}, \mathfrak{so}(2r+1),2s+1)\rightarrow \mathcal{V}^{\dagger}_{\vec{\lambda}^T\cup \vec{\beta}^T}(\mathfrak{X}, \mathfrak{so}(2s+1),2r+1)
\end{equation}
This map $\alpha^{\vee}$ is called the rank-level duality map. The main result of this paper is the following:
\begin{theorem}\label{duality}
The rank-level duality map defined above is an isomorphism.
\end{theorem}
The rest of the paper is devoted to the proof of Theorem ~\ref{duality}. First, we observe that by Proposition ~\ref{diaaut2} and Proposition ~\ref{vacua}, we can reduce the statement of Theorem ~\ref{duality} into the following non-equivalent statements (see Remark 1.2). 
\begin{enumerate}
\item Let $\sum_{i=1}^n|\lambda_i|$ is even. $$\mathcal{V}_{\vec{\lambda}}(\mathfrak{X},\mathfrak{so}(2r+1), 2s+1)\simeq  \mathcal{V}^{\dagger}_{\vec{\lambda}^T}(\mathfrak{X},\mathfrak{so}(2s+1), 2r+1 ).$$ %where $\vec{z}$ denotes $n$-distinct points on $\mathbb{P}^1$.

\item Let $\sum_{i=1}^n|\lambda_i|$ is odd. $$\mathcal{V}_{\vec{\lambda},0}(\mathfrak{X},\mathfrak{so}(2r+1), 2s+1)\simeq  \mathcal{V}^{\dagger}_{\vec{\lambda}^T,\sigma(0)}(\mathfrak{X},\mathfrak{so}(2s+1), 2r+1).$$ %where $\vec{z}$ denotes $n+1$-distinct points on $\mathbb{P}^1$.

\item Let $\sum_{i=1}^n|\lambda_i|$ is even. $$\mathcal{V}_{\vec{\lambda},\sigma(0)}(\mathfrak{X},\mathfrak{so}(2r+1), 2s+1)\simeq \mathcal{V}^{\dagger}_{\vec{\lambda}^T,\sigma(0)}(\mathfrak{X},\mathfrak{so}(2s+1), 2r+1).$$ %where $\vec{z}$ denotes $n+1$-distinct points on $\mathbb{P}^1$.
\end{enumerate}

We will show that the source and the target of the map $\alpha^{\vee}$ in \eqref{rld} has the same dimension. Our strategy for that will be to prove it for the above the above three statements. This will be done in the next section.
\begin{remark} The decomposition of the level one highest weight integrable module $\mathcal{H}_{\omega_d}$ of $\widehat{\mathfrak{so}}(2d+1)$ is given in ~\cite{H}. Furthermore, the decomposition of all level one highest weight integrable modules for the conformal pairs $(B_r, D_s)$ and $(D_r, D_s)$ are given in ~\cite{H}. In all of the above cases, the rank-level duality map is not canonically defined due to failure of some uniqueness properties.  
\end{remark}
\section{Verlinde Formula and equality of dimensions}In this section, we give a complete proof of the equality of dimensions (see Section 7.3) of the source and the target of rank-level duality maps discussed in Section 6. Our key tool is the Verlinde formula for the dimensions of conformal blocks. Another key ingredient is a generalization of a lemma from ~\cite{FH}.

\subsection{Dimensions of some conformal blocks}In this section, we calculate the dimensions of some conformal blocks which we use later in the proof of the rank-level duality. Let $\frg$ be any simple Lie algebra and $\mathfrak{s_{\theta}}$ denote the Lie subalgebra of $\frg$ isomorphic to $\mathfrak{sl}_2$ generated by $H_{\theta}$, $\frg_{\theta}$ and $\frg_{-\theta}$. A $\frg$-module $V$ of level $\ell$ decomposes as a direct sum of $\mathfrak{s}_{\theta}$-modules as follows:
$$V\simeq \oplus _{i=1}^{\ell}V^{i},$$ where $V^{i}$ is a direct sum of $\mathfrak{sl}_2$-modules isomorphic to $\operatorname{Sym}^i\mathbb{C}^2$. We recall the following description of conformal blocks on three pointed $\mathbb{P}^1$ from ~\cite{B}.
\begin{proposition}\label{bea}
Let $\mathfrak{X}$ be the data associated to a three pointed $\mathbb{P}^1$ with chosen coordinates and $\lambda, \mu, \nu \in P_{\ell}(\frg)$. Then, the conformal block $\mathcal{V}^{\dagger}_{\lambda, \mu, \nu}(\mathfrak{X}, \frg)$ is canonically isomorphic to the space of $\frg$-invariant forms $\phi$ on $V_{\lambda}\otimes V_{\mu}\otimes V_{\nu}$ such that $\phi$ restricted to $V_{\lambda}^p\otimes V_{\mu}^q \otimes V_{\nu}^r$ is zero when ever $p+q+r >2\ell$.
\end{proposition}
\subsubsection{The case $\frg=\mathfrak{so}(2r+1)$ at level 1}
Let $\vec{p}=(P_1,P_2,P_3)$ be three distinct points on $\mathbb{P}^1$ with chosen coordinates and $\mathfrak{X}$ be the associated data. The level one dominant integral weights of $\mathfrak{so}(2r+1)$ are $0$, $\omega_1$ and $\omega_r$. Let $\mathcal{V}^{\dagger}_{\lambda_1, \lambda_2, \lambda_3}(\mathfrak{X}, \mathfrak{so}(2r+1))$ denote the conformal block on $\mathbb{P}^1$ with three marked points and weights $\lambda_1, \lambda_2, \lambda_3$ at level one. The following dimensions are calculated in ~\cite{Fa}:
\begin{itemize}
\item $\dim_{\mathbb{C}}\mathcal{V}^{\dagger}_{\omega_1, \omega_1, 0}(\mathfrak{X}, \mathfrak{so}(2r+1),1)=1.$
\item $\dim_{\mathbb{C}}\mathcal{V}^{\dagger}_{\omega_1, \omega_1, \omega_1}(\mathfrak{X}, \mathfrak{so}(2r+1),1)=0.$
\item $\dim_{\mathbb{C}}\mathcal{V}^{\dagger}_{\omega_1, \omega_1, \omega_r}(\mathfrak{X}, \mathfrak{so}(2r+1),1)=0.$
\item $\dim_{\mathbb{C}}\mathcal{V}^{\dagger}_{\omega_1, \omega_r, \omega_r}(\mathfrak{X}, \mathfrak{so}(2r+1),1)=1.$
\end{itemize}
\begin{lemma}
Let $P_1, \dots, P_n$ be $n$ distinct points on $\mathbb{P}^1$ with chosen coordinates and $\mathfrak{X}$ be the associated data. Assume that $\vec{\lambda}=(\omega_1, \dots, \omega_1)$. Then, $\dim_{\mathbb{C}}\mathcal{V}^{\dagger}_{\vec{\lambda}}(\mathfrak{X}, \mathfrak{so}(2r+1),1)=1$, if $n$ is even, and zero if $n$ is odd.
\end{lemma}
\begin{proof}
The proof follows from above and factorization of conformal blocks.
\end{proof}
\subsubsection{The case $\frg=\mathfrak{so}(2r+1)$ at level $\ell$}We calculate the dimensions of some special conformal blocks on three pointed $\mathbb{P}^1$ at any level $\ell$. We first recall the following tensor product decomposition from ~\cite{L}:
\begin{proposition}\label{lr}
Let $\lambda =\sum_{i=1}^{r} a_i\omega_i \in P_{+}^0$. Then,
$$V_{\lambda}\otimes V_{\omega_1}\simeq \oplus_{\gamma}V_{\gamma},$$ where $\gamma$ is of the following form: 
\begin{itemize}
\item If $a_r\neq 0$, then $\gamma$ is either $\lambda$ or obtained from $\lambda$ by adding or deleting a box from the Young diagram of $\lambda$.
\item If $a_r=0$, then $\gamma$ is obtained from $\lambda$ by adding or deleting a box from the Young diagram of $\lambda$.
\end{itemize}

\end{proposition}\label{1dima}
We use the above proposition to calculate the dimensions of the following conformal blocks. 
\begin{proposition}\label{1dim} Let $\lambda=\sum_{i=1}^r a_i\omega_i \in P_{\ell}^0(\mathfrak{so}(2r+1))$. Then, the dimension of the conformal block $\mathcal{V}^{\dagger}_{\lambda,\gamma, \omega_1}(\mathfrak{X}, \mathfrak{so}(2r+1),\ell)$ is one, if  $\gamma \in P_{\ell}^0(\mathfrak{so}(2r+1))$ is of the form in Proposition \ref{lr} and zero otherwise.

\end{proposition}
\begin{proof}
The otherwise part follows from Proposition ~\ref{lr}. Assume that $a_r \neq 0$ and $\gamma$ is either $\lambda$ or obtained from $\lambda$ by adding or deleting a box. For a $\mathfrak{so}(2r+1)$-equivariant form $\phi$ on $V_{\lambda}\otimes V_{\omega_1}\otimes V_{\gamma}$, it's restriction to $V_{\omega_1}^{1}\otimes V_{\lambda}^{\ell}\otimes V_{\gamma}^{\ell}$ is zero, since $\mathbb{C}^2\otimes \operatorname{Sym}^{\ell}\mathbb{C}^2$ does not contain $\operatorname{Sym}^{\ell}\mathbb{C}^2$ as an $\mathfrak{sl}_2(\mathbb{C})$-submodule. Thus, by Proposition ~\ref{bea} and Proposition \ref{lr}, the dimension of $\mathcal{V}^{\dagger}_{\lambda,\gamma, \omega_1}(\mathfrak{X}, \mathfrak{so}(2r+1),\ell)$ is one. The case when $a_r=0$ follows similarly.

\end{proof}

\subsection{Verlinde formula}In this section, we recall the Verlinde formula that calculates dimensions of conformal blocks. First, we start with the Weyl character formula. 
\subsubsection{Weyl character formula}

Here, we first state a basic matrix identity which is an easy generalization of Lemma A.42 from ~\cite{FH}. Suppose $A=(a_{ij})$ is an $(r+s)\times (r+s)$ matrix and $U=(u_1,\dots,u_r)$ and $T=(t_1,\dots,t_r)$ be two sequences of $r$ distinct integers from $\{1,2,\dots, (r+s)\}$. Let $A_{U,T}$ denote the $(r\times r)$ matrix whose $(i,j)$-th entry is $a_{u_i,t_j}$ Similarly define the $(s\times s)$ matrix $B_{T^c,U^c}$, where $U^c$ and $T^c$ are the complements of $U$ and $T$ respectively.
\begin{lemma}\label{surprise}
Let $A$ and $B$ be two $(r+s)\times (r+s)$ matrices whose product is a diagonal matrix $D$. Suppose the $(i,i)$-th entry of $D$ is $a_i$. Let $\pi=(U,U^c)$ and $(T,T^c)$ be permutations of the sequence $(1,\dots, r+s)$, where $|U|=|T|=r$. Then, the following identity of determinants holds:
$$\big(a_{\pi(r+1)}\dots a_{\pi(r+s)}\big)\operatorname{det}A_{U,T}= \operatorname{sgn}(U,U^c)\operatorname{sgn}(T,T^c)\operatorname{det}A\operatorname{det}B_{T^c,U^c}.$$
\end{lemma}

\begin{proof}
Consider the permutation matrices $P$, $Q^{-1}$ associated to the permutation $(U,U^c)$ and $(T,T^c)$ respectively. Then,
\[PAQ=\left(\begin{array}{cc}
A_1 & A_2 \\
A_3 & A_4 \end{array}\right),\ \mbox{where $A_{U,T}=A_1$}\] and similarly 

\[Q^{-1}BP^{-1}=\left(\begin{array}{cc}
B_1 & B_2 \\
B_3 & B_4 \end{array}\right),\ \mbox{where $B_{T^c,U^c}=B_4.$}\] 

Now 
\[\left(\begin{array}{cc}
A_1 & A_2 \\
A_3 & A_4 \end{array}\right)\times \left(\begin{array}{cc}
I_k & B_2\\
0& B_4 \end{array}\right)=\left(\begin{array}{cc}
A_1 & 0\\
A_3 & \Lambda. \end{array}\right)
\] where $\Lambda$ is a diagonal matrix whose $(i,i)$-th entry is $a_{\pi(r+i)}$. Taking determinant of both sides of the above matrix equation, we get the desired equality.
\end{proof}

We are now ready to state the Weyl character formula for $\mathfrak{so}(2r+1)$ following ~\cite{FH}. Let $\mu \in P^0_{2s+1}(\mathfrak{so}(2r+1))$ and $\mu+ \rho=\sum_{i=1}^r u_iL_i$, where $u_i$'s is as defined in Section \ref{branching2}. Let $\lambda=\sum_{i=1}^{r}\lambda^i L_i$ be any dominant integral weight of $\mathfrak{so}(2r+1)$ and $V_{\lambda}$ be the irreducible highest weight module of $\mathfrak{so}(2r+1)$ with weight $\lambda$. Then, by the Weyl character formula 
\begin{equation}\label{WCF}
\operatorname{Tr}_{V_{\lambda}}(\exp{\pi\sqrt{-1}\frac{ \mu+ \rho}{(r+s)}})=\frac{\operatorname{det}\big(\zeta^{u_i(\lambda_j+r-j+\frac{1}{2})}-\zeta^{-{u_i(\lambda_j+r-j+\frac{1}{2})}}\big)}{\operatorname{det}\big(\zeta^{u_i(r-j+\frac{1}{2})}-\zeta^{-{u_i(r-j+\frac{1}{2})}}\big)},
\end{equation}
where $\mu+\rho$ is considered as element of $\mathfrak{h}$ under the identification of $\mathfrak{h}$ with $\mathfrak{h}^*$, $\exp$ is the exponential map from $\mathfrak{so}(2r+1)$ to $\operatorname{SO}_{2r+1}$, $\zeta=\exp{\big(\frac{\pi\sqrt{-1}}{r+s}\big)}$.

\subsubsection{Verlinde Formula}
Let us first recall the Verlinde formula in full generality. Let $C$ be a nodal curve of genus $g$ and $P_1,\dots, P_n$ be $n$ distinct smooth points on $C$ and $\mathfrak{X}$ be the associated data. We fix a Lie algebra $\frg$ and $\vec{\lambda}=(\lambda_1,\dots ,\lambda_n)$-an $n$ tuple of dominant integral weights of $\frg $ of level $\ell$. We refer the reader to ~\cite{B, F11, TUY} for a proof of the following:
\begin{theorem}The dimension of the conformal block $\mathcal{V}^{\dagger}_{\vec{\lambda}}(\mathfrak{X}, \mathfrak{g},\ell)$ is
$$\{(\ell + g^*)^{\operatorname{rank}\frg}|P/Q_{long}|\}^{g-1}\sum_{\mu \in P_{\ell}(\frg)}\operatorname{Tr}_{V_{\vec{\lambda}}}(\exp{2\pi\sqrt{-1}\frac{\mu+\rho}{\ell+g^*}})\prod_{\alpha >0}\bigg| 2\sin{\pi}\frac{(\mu +\rho,\alpha)}{\ell+ g^*}\bigg|^{2-2g},$$ where $\exp$ is the exponential map from $\frg$ to the simply connected Lie group $G$, $Q_{long}$ is the lattice of long roots and $g^*$ is the dual Coxeter number of $\frg$.
\end{theorem}

Let us now specialize to the case $g=0$, $\mathfrak{g}=\mathfrak{so}(2r+1)$, $\ell=2s+1$ and $\vec{\lambda}=(\lambda_1,\dots, \lambda_n)$ an $n$ tuple of weights in $P_{2s+1}^0(\mathfrak{so}(2r+1))$. The dual Coxeter number of $\mathfrak{so}(2r+1)$ is $2r-1$ and $\{(\ell + g^*)^{\operatorname{rank}\frg}|P/Q_{long}|\}=4(k)^r$, where $k=2(r+s)$. Then, by the Weyl character formula \ref{WCF}, we can rewrite the Verlinde formula as follows: 
\begin{equation}\label{nasty1}
\sum_{U\in P_{2s+1}(\mathfrak{so}(2r+1))}\prod_{q=1}^n\frac{\operatorname{det}\big(\zeta^{u_i(\lambda_q^j+r-j+\frac{1}{2})}-\zeta^{-{u_i(\lambda_q^j+r-j+\frac{1}{2})}}\big)}{\operatorname{det}\big(\zeta^{u_i(r-j+\frac{1}{2})}-\zeta^{-{u_i(r-j+\frac{1}{2})}}\big)} {\bigg( \frac{\Phi_{k}(U)}{4k^r} \bigg)},
\end{equation}
where $\mu+\rho=\sum_{i=1}^ru_iL_i$, the set $U=(u_1>u_2>\dots> u_r)$, $\lambda_q=(\lambda_q^1, \lambda_q^2 ,\dots, \lambda_q^r)$ and $\Phi_k(U)$ are as in Section 2 of ~\cite{OW}. We recall the definition of $\Phi_k(U)$ in Section \ref{trigfunc} for completeness. 

\subsection{Equality of dimensions}

\begin{lemma}\label{centertrace}Let $\sigma$ be the non-trivial element of the center of $\operatorname{Spin}(2r+1)$. The element $\sigma$ acts by diagram automorphism on $P_{2s+1}^0(\mathfrak{so}(2r+1))$. 
%If 
%$V_{\lambda_1}\otimes \dots \otimes V_{\lambda_n}$ has non-trivial  $\mathfrak{so}(2r+1)$-invariants. 
Then, $$\operatorname{Tr}_{V_{\vec{\lambda}}}(\exp{\pi\sqrt{-1}\frac{\sigma\mu+\rho}{r+s}})=\operatorname{Tr}_{V_{\vec{\lambda}}}(\exp{\pi\sqrt{-1}\frac{\mu+\rho}{r+s}}),$$
where $\operatorname{exp}$ is the exponential map form $\mathfrak{so}(2r+1)$ to the special orthogonal group $\operatorname{SO}(2r+1)$.
\end{lemma}
\begin{proof}
Let $\mu=\sum_{i=1}^r a_i\omega_i \in P_{2s+1}(\mathfrak{so}(2r+1))$. Then, the weight $\sigma(\mu)$ is given by the formula $(2s+1-2(a_1+\dots +a_r)+a_1+a_r)\omega_1+\sum_{i=2}^ra_i\omega_i$. We calculate the following weight:
\begin{eqnarray*}
\sigma(\mu)+\rho&=& (2s+2-2(a_1+\dots+a_r)+a_1+a_r)\omega_1+ \sum_{i=2}^r(a_i+1)\omega_i,\\
%&=& (2s+1-2(a_1+\dots+a_r)+a_r)\omega_1 + \sum_{i=1}^r(a_i+1)\omega_i\\
%&=& (2s+1-2(a_1+\dots+a_r)+a_r)L_1 + (a_1+\dots + a_{r-1}+(r-1)+\frac{a_r+1}{2})L_1 + \\
%&& ((a_2+ a_3+\dots+a_{r-1})+(r-2)+\frac{a_r+1}{2})L_2 + \dots +\frac{a_r+1}{2}L_r\\
&=& ((2s+1)-(a_1+\dots +a_{r-1})-\frac{a_r}{2}+\frac{2r-1}{2})L_1 + \\
&&((a_2+ a_3+\dots+a_{r-1})+(r-2)+\frac{a_r+1}{2})L_2 + \dots +\frac{a_r+1}{2}L_r.
\end{eqnarray*}
Let $w$ be an element of the Weyl group of $\mathfrak{so}(2r+1)$ which sends $L_1 \rightarrow -L_1$. Then,
\begin{eqnarray*}
w.(\sigma \mu +\rho)&=& (a_1+a_2+\dots +\frac{a_r}{2}-(2s+1)-(2r-1)+ r-\frac{1}{2})L_1 +\\ 
                  &&((a_2+ a_3+\dots+a_{r-1})+(r-2)+\frac{a_r+1}{2})L_2 + \dots +\frac{a_r+1}{2}L_r,\\
                  &=& \mu+\rho -2(r+s)L_1.
\end{eqnarray*}
Thus, we get the following identity:
\begin{eqnarray*}
\exp(2\pi\sqrt{-1}\frac{w.(\sigma\mu+\rho)}{2(r+s)})&=&\exp(2\pi\sqrt{-1}\frac{\mu+\rho}{2(r+s)}).\\
%&=& \exp(2\pi\sqrt{-1}\frac{\mu+\rho}{2(r+s)}).\sigma
\end{eqnarray*}
%Hence the Lemma follows by the Weyl character formula.
%For each $\lambda_i \in P_{2s+1}^0(\mathfrak{so}(2r+1))$, the following equality follows by the Weyl character formula. $$\operatorname{Tr}_{V_{\lambda_i}}(\exp(2\pi\sqrt{-1}\frac{w.(\sigma\mu+\rho)}{2(r+s)}))=\operatorname{Tr}_{V_{\lambda_i}}(\exp(2\pi\sqrt{-1}\frac{(\sigma\mu+\rho)}{2(r+s)})).$$ 
%$$\operatorname{Tr}_{V_{\lambda_i}}(\exp(2\pi\sqrt{-1}L_1)=1$$
%Since the center acts by scalar on $V_{\lambda_1}\otimes \dots \otimes V_{\lambda_n}$, the existance of non-trivial invariants of $V_{\lambda_1}\otimes \dots \otimes V_{\lambda_n}$ guarantees  that the scalar is identity. 
Now the proof follows directly from the Weyl character formula.
\end{proof}
\begin{remark}We refer the reader to ~\cite{B1} for a general discussion of the action of the center of the simply connected group $G$ on $P_{\ell}(\frg)$. The action of the center for all classical Lie algebras can also be found in ~\cite{OW}.
\end{remark}

Consider $\mu \in P^{+}_{2s+1}(\mathfrak{so}(2r+1))$ and $\mu' \in P_{2s+1}^0(\mathfrak{so}(2r+1))$. Let $\mu+\rho=\sum_{i=1}^ru_iL_i$ and $\mu'+\rho'=\sum_{i=1}^r(u_i'-\frac{1}{2})L_i$. Consider the sets $U=(u_1>u_2>\dots>u_r)$ , and $U'=(u_1'>\dots>u_r')$ and let $[U]$ and $[U']$ denote the class of $\mu$, $\mu'$ in $P^{+}_{2s+1}(\mathfrak{so}(2r+1))/\Gamma$ and $P^{0}_{2s+1}(\mathfrak{so}(2r+1))/\Gamma$ respectively.

 Since the function $\Phi_k$ is invariant under the action of center, by Lemma ~\ref{centertrace}, we can rewrite the Verlinde formula in ~\ref{nasty1} as the sum of the following terms:
\begin{enumerate}
\item $$\sum_{[U] \in P^{+}_{2s+1}(\mathfrak{so}(2r+1))/\Gamma}|\operatorname{Orb}_{U}|\prod_{q=1}^n\frac{\operatorname{det}\big(\zeta^{u_i(\lambda_q^j+r-j+\frac{1}{2})}-\zeta^{-{u_i(\lambda_q^j+r-j+\frac{1}{2})}}\big)}{\operatorname{det}\big(\zeta^{u_i(r-j+\frac{1}{2})}-\zeta^{-{u_i(r-j+\frac{1}{2})}}\big)} {\bigg( \frac{\Phi_{k}(U)}{4k^r} \bigg)},$$

\item $$\sum_{[U']\in P^{0}_{2s+1}(\mathfrak{so}(2r+1))/\Gamma}|\operatorname{Orb}_{U'}|\prod_{q=1}^n\frac{\operatorname{det}\big(\zeta^{(u_i'-\frac{1}{2})(\lambda_q^j+r-j+\frac{1}{2})}-\zeta^{-{(u_j'-\frac{1}{2})(\lambda_q^j+r-j+\frac{1}{2})}}\big)}{\operatorname{det}\big(\zeta^{(u_i'-\frac{1}{2})(r-j+\frac{1}{2})}-\zeta^{-{(u_i'-\frac{1}{2})(r-j+\frac{1}{2})}}\big)}{\bigg( \frac{\Phi_{k}(U'-\frac{1}{2})}{4k^r}\big)},$$
\end{enumerate}
where $|\operatorname{Orb}_U|$, $|\operatorname{Orb}_{U'}|$ denote the length of the orbits of $\mu$ and $\mu'$ under the action of $\Gamma$ on $P^{+}_{2s+1}(\mathfrak{so}(2r+1))$ and $P^{0}_{2s+1}(\mathfrak{so}(2r+1))$. The sets $P^{+}_{2s+1}(\mathfrak{so}(2r+1))/\Gamma$ and $P^{0}_{2s+1}(\mathfrak{so}(2r+1))/\Gamma$ denote the orbits of $P^{+}_{2s+1}(\mathfrak{so}(2r+1))$ and $P^{0}_{2s+1}(\mathfrak{so}(2r+1))$ under the action of $\Gamma$ respectively.

\subsection{Final Step of Dimension check}
Let us recall the following two lemmas from ~\cite{OW}. We refer the reader to ~\cite{OW}, Corollary 1.7 and Corollary 1.8 for a proof.
\begin{lemma}\label{trigequality1}

For a positive integer $a$, let $V$ and $V^{c}$ be complementary subsets of $\{1,\dots ,a-1\}$. Then,
$$\frac{(2a)^{|V|}}{\Phi_{2a}(V)}=\frac{2(2a)^{|V^{c}\cup \{a\} |}}{\Phi_{2a}(V^{c}\cup \{a\})}.$$

\end{lemma}

\begin{lemma}\label{trigequality2}
Let $V'\subset S=\{\frac{1}{2}, \dots a-\frac{1}{2}\}$ and ${V'}^{c}$ be the complement. Then, we have:
$$\frac{(2a)^{|V'|}}{\Phi_{2a}(V')}=\frac{(2a)^{V'^{c}}}{\Phi_{2a}(a-V'^{c})}.$$

\end{lemma}
Let $\lambda_i\in \mathcal{Y}_{r,s}$ such that $\sum_{i=1}^n|\lambda_i|$ is even and $\mathfrak{X}$ be the data associated to $n$ distinct points on $\mathbb{P}^1$ with chosen coordinates. Denote by $\vec{\lambda}$, an $n$ tuple of weights $ (\lambda_1,\dots,\lambda_n)$ and $\vec{\lambda}^T$ the $n$ tuple of weights $(\lambda_1^T,\dots,\lambda_n^T)$. Consider the conformal blocks $\mathcal{V}^{\dagger}_{\vec{\lambda}}(\mathfrak{X},\mathfrak{so}(2r+1),2s+1)$ and $\mathcal{V}^{\dagger}_{\vec{\lambda}^T}(\mathfrak{X},\mathfrak{so}(2s+1),2r+1)$. 
\begin{proposition}\label{dim1}If $\sum_{i=1}^n|\lambda|$ is even, then the following equality of dimensions holds:
% the conformal blocks are non-zero, then 
$$\dim_{\mathbb{C}}\mathcal{V}^{\dagger}_{\vec{\lambda}}(\mathfrak{X},\mathfrak{so}(2r+1),2s+1)=\dim_{\mathbb{C}}\mathcal{V}^{\dagger}_{\vec{\lambda}^T}(\mathfrak{X},\mathfrak{so}(2s+1), 2r+1).$$

\end{proposition}
\begin{proof}
By Lemma ~\ref{orbitbijection}, it is enough to show that the following equalities hold
\begin{eqnarray*}
&&|\operatorname{Orb}_{U}|\prod_{q=1}^n\frac{\operatorname{det}\big(\zeta^{u_i(\lambda_q^j+r-j+\frac{1}{2})}-\zeta^{-{u_i(\lambda_q^j+r-j+\frac{1}{2})}}\big)}{\operatorname{det}\big(\zeta^{u_i(r-j+\frac{1}{2})}-\zeta^{-{u_i(r-j+\frac{1}{2})}}\big)} {\bigg( \frac{\Phi_{k}(U)}{4k^r} \bigg)}\\
&&=|\operatorname{Orb}_{U^c}|\prod_{q=1}^n\frac{\operatorname{det}\big(\zeta^{u^c_i((\lambda_q^T)^j+r-j+\frac{1}{2})}-\zeta^{-{u^c_i((\lambda^T_q)^j+r-j+\frac{1}{2})}}\big)}{\operatorname{det}\big(\zeta^{u^c_i(r-j+\frac{1}{2})}-\zeta^{-{u^c_i(r-j+\frac{1}{2})}}\big)} {\bigg( \frac{\Phi_{k}(U^c)}{4k^s} \bigg)},
\end{eqnarray*}
where $U=\{u_1>\dots >u_r)\}\in P^{+}_{2s+1}(\mathfrak{so}(2r+1))/\Gamma$ and $r+s\in U$ and $u_i^c$ is same as in Section 5.
\begin{eqnarray*}
&&|\operatorname{Orb}_{U'}|\prod_{q=1}^n\frac{\operatorname{det}\big(\zeta^{(u_i'-\frac{1}{2})(\lambda_q^j+r-j+\frac{1}{2})}-\zeta^{-{(u_i'-\frac{1}{2})(\lambda_q^j+r-j+\frac{1}{2})}}\big)}{\operatorname{det}\big(\zeta^{(u_i'-\frac{1}{2})(r-j+\frac{1}{2})}-\zeta^{-{(u_i'-\frac{1}{2})(r-j+\frac{1}{2})}}\big)} {\bigg( \frac{\Phi_{k}(U'-\frac{1}{2})}{4k^r} \bigg)}\\
&=&|\operatorname{Orb}_{((r+s+1)-U'^c)}|\times\\ &&\prod_{q=1}^n\frac{\operatorname{det}\big(\zeta^{(u''_i-\frac{1}{2})((\lambda_q^T)^j+r-j+\frac{1}{2})}-\zeta^{-{(u''_i-\frac{1}{2})((\lambda_q^T)^j+r-j+\frac{1}{2})}}\big)}{\operatorname{det}\big(\zeta^{(u''_i-\frac{1}{2})(r-j+\frac{1}{2})}-\zeta^{-{(u''_i-\frac{1}{2})(r-j+\frac{1}{2})}}\big)} {\bigg( \frac{\Phi_{k}((r+s+\frac{1}{2})-U'^c)}{4k^s} \bigg)},
\end{eqnarray*}
where $U'=\{u_1'>u_2'>\dots > u_r'\}\in P^0_{2s+1}(\mathfrak{so}(2r+1))/\Gamma$ and $u_i''$ is same as defined in Section 5.
 Now by Lemma ~\ref{trigequality1} and Lemma ~\ref{trigequality2}, we know that 
$$|\operatorname{Orb}_{U}|{\bigg( \frac{\Phi_{k}(U)}{4k^r} \bigg)}=|\operatorname{Orb}_{U^c}|{\bigg( \frac{\Phi_{k}(U^c)}{4k^s} \bigg)}.$$
$${\bigg( \frac{\Phi_{k}(U'-\frac{1}{2})}{4k^r} \bigg)}={\bigg( \frac{\Phi_{k}((r+s+\frac{1}{2})-U'^c)}{4k^s} \bigg)}.$$
We are reduced to show the following identity of determinants for the pair $(U,U^c)$:
$$\prod_{q=1}^n\frac{\operatorname{det}\big(\zeta^{u_i(\lambda_q^j+r-j+\frac{1}{2})}-\zeta^{-{u_i(\lambda_q^j+r-j+\frac{1}{2})}}\big)}{\operatorname{det}\big(\zeta^{u_i(r-j+\frac{1}{2})}-\zeta^{-{u_i(r-j+\frac{1}{2})}}\big)}=\prod_{q=1}^n\frac{\operatorname{det}\big(\zeta^{u^c_i((\lambda_q^T)^j+r-j+\frac{1}{2})}-\zeta^{-{u^c_i((\lambda_q^T)^j+r-j+\frac{1}{2})}}\big)}{\operatorname{det}\big(\zeta^{u^c_i(r-j+\frac{1}{2})}-\zeta^{-{u^c_i(r-j+\frac{1}{2})}}\big)}.$$ 
This follows directly from Lemma ~\ref{dirty1}. We also need to show the following equality of determinants for the pair $(U',U'^c)$ and $\lambda^T_q=({(\lambda^T_q)}^1\geq \dots \geq {(\lambda^T_q)}^s)$.
\begin{eqnarray*}
&&\prod_{q=1}^n\frac{\operatorname{det}\big(\zeta^{(u_i'-\frac{1}{2})(\lambda_q^j+r-j+\frac{1}{2})}-\zeta^{-{(u_i'-\frac{1}{2})(\lambda_q^j+r-j+\frac{1}{2})}}\big)}{\operatorname{det}\big(\zeta^{(u_i'-\frac{1}{2})(r-j+\frac{1}{2})}-\zeta^{-{(u_i'-\frac{1}{2})(r-j+\frac{1}{2})}}\big)}\\
 &&= \prod_{q=1}^n\frac{\operatorname{det}\big(\zeta^{(u''_i-\frac{1}{2})((\lambda_q^T)^j+r-j+\frac{1}{2})}-\zeta^{-{(u''_i-\frac{1}{2})((\lambda_q^T)^j+r-j+\frac{1}{2})}}\big)}{\operatorname{det}\big(\zeta^{(u''_i-\frac{1}{2})(r-j+\frac{1}{2})}-\zeta^{-{(u''_i-\frac{1}{2})(r-j+\frac{1}{2})}}\big)},
\end{eqnarray*}

This follows from Lemma ~\ref{dirty2}. 
\end{proof}
With the same notation and assumptions as in Proposition ~\ref{dim1}, we have the following proposition. 
\begin{proposition}If $\sum_{i=1}^n|\lambda|$ is even, then %Assume that both the conformal blocks are non-zero, then 
$$\dim_{\mathbb{C}}\mathcal{V}^{\dagger}_{\vec{\lambda}\cup\sigma (0)}(\mathfrak{X},\mathfrak{so}(2r+1),2s+1)=\dim_{\mathbb{C}}\mathcal{V}^{\dagger}_{\vec{\lambda}^T\cup\sigma( 0)}(\mathfrak{X},\mathfrak{so}(2s+1),2r+1).$$
\end{proposition}\label{dim2}

\begin{proof}
The proof follows from the proof of Proposition ~\ref{dim1}, Lemma ~\ref{trivial1} and Lemma ~\ref{trivial2}
\end{proof}

For each $1\leq i \leq n$, let $\lambda_i\in \mathcal{Y}_{r,s}$ be such that $\sum_{i=1}^n|\lambda_i|$ is odd. Let $\mathfrak{X}$ be the data associated to $n$ distinct points on $\mathbb{P}^1$ with chosen coordinates. Denote by $\vec{\lambda}$ the $n$ tuple of weights $ (\lambda_1,\dots,\lambda_n)$ and $\vec{\lambda}^T$ the $n$ tuple of weights $(\lambda_1^T,\dots,\lambda_n^T)$. Consider the conformal blocks $\mathcal{V}^{\dagger}_{\vec{\lambda}\cup 0}(\mathfrak{X},\mathfrak{so}(2r+1),2s+1)$ and $\mathcal{V}^{\dagger}_{\vec{\lambda}^T\cup\sigma (0)}(\mathfrak{X},\mathfrak{so}(2s+1),2r+1)$. Then, we have the following equality of dimensions:
\begin{proposition}\label{dim3}
%Asssume that both the conformal blocks are non zero then 
$$\dim_{\mathbb{C}}\mathcal{V}^{\dagger}_{\vec{\lambda}\cup 0}(\mathfrak{X},\mathfrak{so}(2r+1),2s+1)=\dim_{\mathbb{C}}\mathcal{V}^{\dagger}_{\vec{\lambda}^T\cup\sigma (0)}(\mathfrak{X},\mathfrak{so}(2s+1),2r+1).$$
\end{proposition}

\begin{remark}
These equalities of the dimensions of the conformal blocks give rise to some new interesting relations between the fusion ring of $\operatorname{SO}(2r+1)$ at level $2s+1$ with the fusion ring of $\operatorname{SO}(2s+1)$ at level $2r+1$.
\end{remark}

\section{Highest Weight Vectors}
In this section, we briefly summarize the construction of level one highest weight integrable modules $\mathcal{H}_{0}(\mathfrak{so}(2r+1))$ and $\mathcal{H}_{1}(\mathfrak{so}(2r+1))$ using Clifford algebras. We use this to explicitly describe the highest weight vectors (see Section 8.2) of the components that appear in the branching. Our discussions closely follow the discussions in ~\cite{H}.

\subsection{Spin modules} We first recall the definition of Clifford algebra. Let $W$ be a vector space (not necessarily finite dimensional) with a non-degenerate bilinear form $\{,\}$.

\begin{definition} We define the {\em Clifford Algebra }associated to $W$ and $\{,\}$ to be $$C(W):=T(W)/ I,$$ where $T(W)$ is the tensor algebra of $W$ and $I$ is the two sided ideal generated by elements of the form $v\otimes w+ w\otimes v -\{v,w\}$. 
\end{definition}
\subsubsection{Spin module of $C(W)$}
Suppose there exists an isotropic decomposition $W=W^{+}\oplus W^{-},$ i.e. $\{ W^{\pm}, W^{\pm}\}=0$ and $\{,\}$ restricted to $W^{+}\oplus W^{-}$ is non-degenerate. Then, the exterior algebra $\bigwedge W^{-}$ can be viewed as a $\bigwedge W^{-}$-module by taking wedge product on the left. This gives rise to the structure of an irreducible $C(W)$-module on $\bigwedge W^{-}$ by defining $$w^{+}.1=0,$$ for all $w^{+}\in W^{+}$ and $1\in \bigwedge W^{-}$. 

Next if $W=W'\oplus \mathbb{C}e$ is an orthogonal direct sum with $\{e,e\}=1$ and $W'$ has an isotropic decomposition of the form $W^{+}\oplus W^{-}$(we refer this as quasi-isotropic decomposition of W). Then, the $C(W')$-module $\bigwedge W^{-}$ described above becomes an irreducible $C(W)$-module by the following action: 
$$\sqrt{2}e.v:=\pm(-1)^pv \ \ \mbox{for} \ \ v\in \bigwedge^{p}W^{-}.$$

Any element of $W^{-}$(respectively $W^{+}$) is called a creation operator (respectively annihilation operator).

\subsubsection{Root Spaces and basis of $\mathfrak{so}(2r+1)$}
Consider a finite dimensional vector space $W_r$ of dimension $2r+1$ with a non-degenerate symmetric bilinear form $\{,\}$. Let $\{e_{i}\}_{i=-r}^r$ be an orthonormal basis of $W_r$. For $j>0$, we set
$$\phi^{j}=\frac{1}{\sqrt{2}}(e_j+\sqrt{-1}e_{-j}); \ \  \phi^{-j}=\frac{1}{\sqrt{2}}(e_j-\sqrt{-1}e_{-j})\hspace{.25cm}\mbox{and} \hspace{.25cm} \phi^0=e_0.$$
Let $\phi^1,\dots, \phi^r, \phi^0, \phi^{-r}, \dots, \phi^{-1}$ be the chosen ordered basis of $W_r$. For any $i,j$, we define 
$E^{i}_j(\phi^{k}):=\delta_{k,j}\phi^{i}$. 

 We identify the Lie algebra $\mathfrak{so}(2r+1)(W_r)$ with $\mathfrak{so}(2r+1)$ as follows:
$$\mathfrak{so}(2r+1):=\{A\in \mathfrak{sl}(2r+1)| A^T J + JA=0\},$$ where $J$ is the following $(2r+1)\times (2r+1)$ matrix:
$$J=\bordermatrix{\text{}& 1&\ldots&r &0&-r& \ldots &-1\cr
                      1  &   &   &  & &  & & 1 \cr
                   \vdots&   &  {\Large{0}}  &  & &  &  1    & \cr
                       r &   &     &  & &  1&        & \cr
                       0 &   &     &  & 1&  &        & \cr
                       -r &   &     &1  & &  &        & \cr
                   \vdots &   & 1  &  & &  &   {\Large{0}}    & \cr
                     -1 & 1  &     &  & &  &        &\cr}.$$
                       
We put $B^{i}_{j}=E^{i}_{j}-E^{-j}_{-i}$ and take the Cartan subalgebra $\mathfrak{h}$ to be the subalgebra of diagonal matrices. Clearly, 
$\mathfrak{h}=\oplus_{j=1}^r\mathbb{C}B^{j}_{j}$. The corresponding dual basis of $\mathfrak{h}^*$ is $L_j$, where $L_j(B^{k}_k)=\delta_{j,k}$. The simple positive roots $\{\alpha_i\}_{i=1}^r$ of $\mathfrak{so}(2r+1)$ are given by $L_1-L_2, \dots, L_{r-1}-L_r, L_r$. The root spaces of $\mathfrak{so}(2r+1)$ are of the form $\frg_{L_i\pm L_j}=\mathbb{C}B^{i}_{\mp j}$ and $\frg_{L_i}=\mathbb{C}B^{i}_0$.

\begin{remark}
The basis of the vector space $W_r$ chosen here is different than the basis in ~\cite{FH}. In this section, we prefer this basis as the branching formulas that we describe in the next section become simpler to state with this new notation. 
\end{remark}
\subsubsection{Spin module $\bigwedge W_{r}^{\mathbb{Z}+\frac{1}{2},-}$ of $\widehat{\mathfrak{so}}(2r+1)$}Consider as before $W_r$ to be a $2r+1$ dimensional complex vector space with a non-degenerate symmetric bilinear form $\{,\}$. Let $$W_{r}^{\pm}= \oplus_{i=1}^r\mathbb{C}\phi^{\pm}.$$ A quasi-isotropic decomposition of $W_r$ given by the following: 
$$W_r=W_r^{+}\oplus W_r^{-}\oplus \mathbb{C}\phi^0.$$
We define a new vector space $W_{r}^{\mathbb{Z}+\frac{1}{2}}$ with an inner product $\{,\}$ as follows:
$$W_{r}^{\mathbb{Z}+\frac{1}{2}}:=W_{r}\otimes t^{\frac{1}{2}}\mathbb{C}[t,t^{-1}]\hspace{.25cm} \mbox{with} \hspace{.25cm} \{w_1(a), w_2(b)\}=\{w_1,w_2\}\delta_{a+b,0},$$ where $w_1,w_2 \in W_r$; $a, b \in \mathbb{Z}+\frac{1}{2}$ and $w_1(a)=w_1\otimes t^a$. We choose a quasi-isotropic decomposition of $W_{r}^{\mathbb{Z}+\frac{1}{2}}$ given as follows:
$$W_{r}^{\mathbb{Z}+\frac{1}{2}}=W_{r}^{\mathbb{Z}+\frac{1}{2},+}\oplus W_{r}^{\mathbb{Z}+\frac{1}{2},-},$$ where 
$W_{r}^{\mathbb{Z}+\frac{1}{2},\pm}:=W_r\otimes t^{\pm \frac{1}{2}}\mathbb{C}[t^{\pm 1}]$.
We define the normal order ${}_o^o{\hspace{.5cm}}^o_{o}$ for $w_1(a), w_2(b)\in W_{r}^{\mathbb{Z}+\frac{1}{2}}$ by the following formula:
\[
{}_o^o{w_1(a)w_2(b)}^o_o = \left\{
\begin{array}{l l}
  -w_2(b)w_1(a) & \quad \text{if $ a>0>b$,}\\
  \frac{1}{2}(w_1(a)w_2(b) - w_2(b)w_1(a)) & \quad \text{if $a=b=0$,}\\
   w_1(a)w_2(b) & \quad \text{otherwise.}\\
\end{array} \right.
\]

We now describe the action of $\widehat{\mathfrak{so}}(2r+1)$ on $\bigwedge W_{r}^{\mathbb{Z}+\frac{1}{2},-}$ and explicitly describe the level one $\widehat{\mathfrak{so}}(2r+1)$-modules $\mathcal{H}_{0}(\mathfrak{so}(2r+1))$ and $\mathcal{H}_{1}(\mathfrak{so}(2r+1))$. For a proof, we refer the reader to ~\cite{Fa}.
\begin{proposition}\label{ff1}
The following map is a Lie algebra monomorphism:
\begin{eqnarray*}
\widehat{\mathfrak{so}}(2r+1)&\rightarrow& \operatorname{End}(\bigwedge W_{r}^{\mathbb{Z}+\frac{1}{2},-}),\\
B^{i}_{j}(m)&  \rightarrow & \sum_{a+b=m}{}^{0}_{0}\phi^{i}(a)\phi^{-j}(b){}^{0}_{0},\\
 c & \rightarrow & \operatorname{id}.
\end{eqnarray*}
\end{proposition} 
\begin{proposition}Suppose $r \geq 1$, then the following are isomorphic as level one $\widehat{\mathfrak{so}}(2r+1)$-modules:
\begin{enumerate}
\item $\bigwedge^{even}(W_{r}^{\mathbb{Z}+\frac{1}{2},-})\simeq \mathcal{H}_{0}(\mathfrak{so}(2r+1)),$

\item $\bigwedge^{odd}(W_{r}^{\mathbb{Z}+\frac{1}{2},-})\simeq \mathcal{H}_{1}(\mathfrak{so}(2r+1)).$
\end{enumerate}
The highest weight vectors are given by $1$ and $\phi^{1}(-\frac{1}{2}).1$ respectively.
\end{proposition}

\subsection{Highest weight vectors}
Let $W_s$ be a $2s+1$ dimensional vector space over $\mathbb{C}$ with a non-degenerate bilinear form $\{,\}$, and let $\{e_{p}\}_{p=1}^s$ be an orthonormal basis of $W_s$. Let $\phi^1,\dots, \phi^s,\phi^0,\phi^{-s},\dots,\phi^{-1}$ be an ordered isotropic basis of $W_s$. The tensor product of $W_d=W_r\otimes W_s$ carries a non-degenerate symmetric bilinear form $\{,\}$ given by the product of the forms on $W_r$ and $W_s$. Clearly the elements $\{e_{j,p}:=e_j\otimes e_p | -r\leq j\leq r \ \mbox{and}\  -s\leq p\leq s \}$ form an orthonormal basis of $W_d$. By $(j,p)>0$, we mean $j>0$ or $j=0, p>0$ and put 
$$\phi^{j,p}=\frac{1}{\sqrt{2}}(e_{j,p}-\sqrt{-1}e_{-j,-p}), \hspace{1cm} \phi^{-j,-p}=\frac{1}{\sqrt{2}}(e_{j,p}+\sqrt{-1}e_{-j,-p}),$$ 
 for $(j,p)>0$. The form $\{,\}$ on $W_d$ is given by the formula 
$$\{ \phi^{j,p},\phi^{-k,-q}\}=\delta_{j,k}\delta_{p,q}, \ \ \mbox{for} \ -r\leq j,k \leq r \ \mbox{and} \ -s\leq p,q \leq s.$$ Let as before $W_d^{\pm}=\bigoplus_{(j,p)>0}\mathbb{C}\phi^{\pm j,\pm p}$ and $\phi^{0,0}=e_{0,0}$. The quasi-isotropic decomposition of $W_d$ is given as follows: 
$$W_d=W_d^{+}\oplus W_d^{-}\oplus \mathbb{C}\phi^{0,0}.$$

We define the operators $E^{j,p}_{k,q}$ by 
$E^{j,p}_{k,q}(\phi^{i,l})=\delta_{i,k}\delta_{l,q}\phi^{j,p}$. We put $$B^{j,p}_{k,q}=E^{j,p}_{k,q}-E^{-k,-q}_{-j,-p}.$$ Consider the Cartan subalgebra $\mathfrak{H}$ of $\mathfrak{so}(2d+1)$ to be the subalgebra generated by the diagonal matrices. Clearly $\mathfrak{H}=\oplus_{(j,p)>0}\mathbb{C}B^{j,p}_{j,p}$. Let $\{L_{j,p}\}$ for $(j,p)>0$ be a dual basis. Thus $\mathfrak{H}^*=\oplus_{(j,p)>0}\mathbb{C}L_{j,p}.$
%Let $L:\mathfrak{so}(2r+1)\rightarrow \mathfrak{so}(2d+1)$ and $R:\mathfrak{so}(2s+1)\rightarrow \mathfrak{so}(2d+1)$ be the of Lie algebras and $\widehat{L}$ and $\widehat{R}$ be the maps for the corresponding affine Lie algebras. 
\subsubsection{Highest weight vectors as wedge product}
To every Young diagram in $\mathcal{Y}_{r,s}$, we associate an $(2r+1)\times (2s+1)$ matrix as follows. First, to every Young diagram $\lambda$, we associate a $(r\times s)$ matrix $Y(\lambda)$ as follows:
\[
Y(\lambda)_{i,j}=\left\{
\begin{array}{l l}
0&\quad \text{if $\lambda$ has a box in the $(i,j)$-th position,} \\
1&\quad \text{otherwise.}\\
\end{array}\right.
\]
Finally to $Y(\lambda)$, we associate the following matrix:
$$\widetilde{Y}(\lambda)=\bordermatrix{\text{}& 1&\ldots&s &0&-s& \ldots &-1\cr
                      1  &   &     &  & 1&  1&\ldots       &1 \cr
                   \vdots&   &  Y(\lambda)   &  &\vdots & \vdots & \ldots      &\vdots \cr
                       r &   &     &  & 1& 1 &            \ldots     &1 \cr
                       0 & 1  &  \ldots   & 1 &1 &1  &\ldots       &1 \cr
                       -r & \vdots  &  \vdots    & \vdots &\vdots  &\vdots   &  \vdots      &\vdots  \cr
                   \vdots &\vdots   &   \vdots  &\vdots  &\vdots  & \vdots & \vdots       & \vdots \cr
                     -1 &  1 &  1   & 1 &1 &1  &   1     &1 \cr}.$$

For $\lambda \in \mathcal{Y}_{r,s}$, let $\widetilde{Y}(\lambda)$ be the image of ${Y}(\lambda)$. We define the following operations on $\widetilde{Y}(\lambda)$ which produces a new matrix:
\begin{eqnarray}
\sigma^{L}(\widetilde{Y}(\lambda))_{j,p}&:=&\widetilde{Y}(\lambda)_{j,p}-\delta_{j,1}\delta_{\widetilde{Y}_{1,|p|},1},\\
\sigma^{R}(\widetilde{Y}(\lambda))_{j,p}&:=&\widetilde{Y}(\lambda)_{j,p}-\delta_{p,1}\delta_{\widetilde{Y}_{|j|,1},1}.
\end{eqnarray}
The following proposition in ~\cite{H} give highest weight vectors for the branching rules described in Section ~\ref{branching2}.
\begin{proposition}\label{highestweight} The vector $\big(\bigwedge_{\widetilde{y}_{j,p}=0}\phi^{j,p}(-\frac{1}{2})\big).1$ well defined up to a sign for each of the matrices $\widetilde{Y}(\lambda)$, $\sigma^{L}(\widetilde{Y}(\lambda))$, $\sigma^{R}(\widetilde{Y}(\lambda))$, $\sigma^L(\sigma^{R}(\widetilde{Y}(\lambda)))$ gives a highest weight vector of the components with highest weight $(\lambda, \lambda^T)$, $(\sigma(\lambda),\lambda^T)$, $(\lambda, \sigma(\lambda^T))$ and $(\sigma(\lambda),\sigma(\lambda^T))$
\end{proposition}
Next, we describe the highest vectors for some of the components in the ``Kac-Moody" form. We use this explicit descriptions to prove the basic cases of the rank-level duality.
\subsubsection{Highest weight vectors in Kac-Moody form}
Let $\lambda,\lambda' \in \mathcal{Y}_{r,s}$ and assume that $\lambda$ is obtained from $\lambda' \in \mathcal{Y}_{r,s}$ by adding two boxes. In terms of the matrices described in Section 8.2.1, $Y(\lambda)$ is obtained from $Y(\lambda')$ by changing $1$ to $0$ in exactly two places of $Y(\lambda')$, say at $(a,b)$ and $(c,d)$. Assume that $(a,b)<(c,d)$ under the lexicographic ordering.

\begin{remark}
Let $V_{\lambda}$ be the finite dimensional $\frg$-module inside $\mathcal{H}_{\lambda}(\frg)$, where $\frg$ is a finite-dimensional semisimple Lie algebra. Every finite dimensional irreducible representation of $\frg$ has a lowest weight vector $v^{\lambda}$. This vector is a highest weight vector for the affine Lie algebra $\widehat{\frg}$ if we had chosen the opposite Borel as the Borel for $\frg$. We call the vector $v^{\lambda}$ as the opposite highest weight vector of $\mathcal{H}_{\lambda}(\frg)$.

\end{remark}
%As described in the branching rule in Section ~\ref{branching2} if $|\lambda|$ is even (respectively odd) then $\mathcal{H}_{\lambda}(\mathfrak{so}(2r+1))\otimes \mathcal{H}_{\lambda^T}(\mathfrak{so}(2s+1))$ appear in the branching of $\mathcal{H}_{0}(\mathfrak{so}(2d+1))$(respectively $\mathcal{H}_{\omega_1}(\mathfrak{so}(2d+1))$). 
The following proposition describe highest weight vectors in the ``Kac-Moody" form, i.e. as elements of universal enveloping of $\widehat{\mathfrak{so}}(2d+1)$ acting on the highest weight vectors of $\mathcal{H}_{0}(\mathfrak{so}(2d+1))$ and $\mathcal{H}_{1}(\mathfrak{so}(2d+1))$.  
\begin{proposition}\label{Kacmoody}Let $\lambda$ and $\lambda'$ be as before. Then, the following holds:
\begin{enumerate}
\item  If $v_{\lambda'}\in \operatorname{End}(\bigwedge W_{d}^{\mathbb{Z}+\frac{1}{2},-})$ is the highest weight vector of the component $\mathcal{H}_{\lambda'}(\mathfrak{so}(2r+1))\otimes \mathcal{H}_{\lambda'^T}(\mathfrak{so}(2s+1))$, then the highest weight vector $v_{\lambda}$ of the component $\mathcal{H}_{\lambda}(\mathfrak{so}(2r+1))\otimes \mathcal{H}_{\lambda^T}(\mathfrak{so}(2s+1))$ is given by the following:
$$v_{\lambda}=B^{a,b}_{-c,-d}(-1).v_{\lambda'}.$$

%\begin{remark}
\item If $v^{\lambda'}$ is the opposite highest weight vector of $\mathcal{H}_{\lambda'}(\mathfrak{so}(2r+1))\otimes \mathcal{H}_{\lambda'^T}(\mathfrak{so}(2s+1))$, then the opposite highest weight vector $v^{\lambda}$ of the component $\mathcal{H}_{\lambda}(\mathfrak{so}(2r+1))\otimes \mathcal{H}_{\lambda^T}(\mathfrak{so}(2s+1))$ is:
$$ v^{\lambda}=B^{-a,-b}_{c,d}(-1).v^{\lambda'}.$$
%\end{remark}
\end{enumerate}
\end{proposition}
\begin{proof}
The proof of the above easily follows from Proposition ~\ref{highestweight} and Proposition ~\ref{ff1}. 
\end{proof}
\begin{remark}
There is no uniqueness in building a Young diagram $\lambda$ starting from the empty Young diagram. So there is no uniqueness in the expressions of the highest weight vectors described in Proposition ~\ref{Kacmoody}.
\end{remark}

\section{Proof of Theorem ~\ref{duality}}\label{mainproof}In this section, we give a proof of Theorem ~\ref{duality}. The main steps of the proof are summarized below.

\subsection{Key steps}The strategy of the proof of Theorem ~\ref{duality} closely follows ~\cite{A} and ~\cite{NT} but has some significant differences in the individual steps.
\subsubsection{Step I}We study the degeneration of rank-level duality maps on $\mathbb{P}^1$ with $n$ marked points. We use Proposition \ref{geometricinput} to reduce to the case for conformal blocks on $\mathbb{P}^1$ with three marked points and the representation attached to one of the marked points is $\omega_1$. The details of this step are explained in Section \ref{generalal}.

\subsubsection{Step II}We are now reduced to proving rank-level dualities for admissible pairs (see Definition \ref{admiss}) of the form $((\omega_1,\lambda_2,\lambda_3), (\omega_1,\beta_1,\beta_2))$. We use Proposition \ref{1dim} to determine which conformal blocks on $\mathbb{P}^1$ with three marked points with representations of the form $(\omega_1,\lambda_2,\lambda_3)$ are non-zero. 

\subsubsection{Step III}We use Proposition \ref{diaaut2} and further reduce to proving rank-dualities for three pointed curves and admissible pairs of the following forms:
\begin{enumerate}
\item  $(\omega_1,\lambda_2,\lambda_3), (\omega_1,\lambda_2^T,\lambda_3^T)$, where $\lambda_2, \lambda_3 \in \mathcal{Y}_{r,s}$ and $\lambda_2$ is obtained from $\lambda_3$ either by adding or deleting a box. Rank-level dualities for these cases are proved in Section \ref{minimalcases}.

\item $(\omega_1, \lambda, \lambda), (\omega_1,\lambda^T, \sigma(\lambda^T))$, where $\lambda \in \mathcal{Y}_{r,s}$ and $(\lambda, L_r)\neq 0$. 
These rank-level dualities are proved in Section \ref{reminimalcases}.
\end{enumerate}

%\subsubsection{Step II}We now reduced to the case when the conformal block is $\mathbb{P}^1$ with $3$ marked points 
\subsection{The minimal three point cases}\label{minimalcases}
In this section, we prove rank-level dualities for some special one dimensional conformal blocks on $\mathbb{P}^1$ with three marked points. We use these cases to prove the rank-level duality isomorphism in the general case. 

The finite dimensional irreducible $\mathfrak{so}(2d+1)$-module $V_{\omega_1}$ can be realized inside $\bigwedge^{odd}W_d^{\mathbb{Z}+\frac{1}{2},-}$ as linear span of vectors of the form $\phi^{i,j}(-\frac{1}{2})$. On $V_{\omega_1}$, there is a canonical $\mathfrak{so}(2d+1)$ invariant bilinear form $Q$ given by the following formula:
$$Q(\phi^{j,p}(-\frac{1}{2}),\phi^{-k,-q}(-\frac{1}{2}))=\delta_{j,k}\delta_{p,q}.$$ For notational convenience, we write $\phi^{i,j}(-\frac{1}{2})$ as $v^{i,j}$.

Throughout this section, we will assume that $(P_1,P_2, P_3)=(1,0,\infty)$ with coordinates $\xi_1=z-1$, $\xi_2=z$ and $\xi_3=\frac{1}{z}$, where $z$ is a global coordinate on $\mathbb{C}$. We denote the associated data by $\mathfrak{X}$. Let $\lambda_2,\lambda_3 \in \mathcal{Y}_{r,s}$, $\vec{\lambda}=(\omega_1,\lambda_2,\lambda_3)$, $\vec{\lambda}^T=(\omega_1,\lambda_2^T, \lambda_3^T)$, $\vec{\Lambda}=(\omega_1,\omega_1,0)$ and $\lambda_2$ is obtained from $\lambda_3$ by adding or deleting a box. 
\begin{remark}
%The following strategy is inspired from a proof of result of B. Kostant. We refer the reader to Proposition 3.2 in ~\cite{SK1} for more details.
The following strategy is influenced by the proof of Proposition 6.3 in ~\cite{A}.
\end{remark}
 Let us summarize our main steps to prove these minimal cases. Let $\langle {\Psi'}| \in \mathcal{V}^{\dagger}_{\vec{\Lambda}}(\mathfrak{X},\mathfrak{so}(2d+1),1)$ be a non-zero element. It is enough to produce $|\Phi_1\otimes \Phi_2 \otimes \Phi_3\rangle  \in \mathcal{H}_{\vec{\lambda}}(\mathfrak{so}(2r+1))\otimes \mathcal{H}_{\vec{\lambda}^T}(\mathfrak{so}(2s+1))$ such that 
$$\langle {\Psi'} | \Phi_1\otimes \Phi_2 \otimes \Phi_3\rangle \neq 0.$$
\subsubsection{Step I} We always choose $|\Phi_2\rangle$( respectively $|\Phi_3\rangle $) to be the highest (respectively opposite highest) weight vector of the integrable module with highest weight ($\lambda_2,\lambda_2^T$) (respectively ($\lambda_3, \lambda_3^T$)). 

\subsubsection{Step II} If $\lambda_3$ is obtained from $\lambda_2$ by adding a box in the $(a,b)$-th coordinate, then we choose $|\Phi_1\rangle$ to be $v^{a,b}$. If $\lambda_2$ is obtained from $\lambda_3$ by adding a box in the $(a,b)$-th coordinate, then we choose $|\Phi_1\rangle $ to be $v^{-a,-b}$. With this choice, it is clear that the $\mathfrak{H}$-weight of $|\Phi_1\otimes \Phi_2\otimes \Phi_3\rangle $ is zero.

\subsubsection{Step III}We use induction on $\operatorname{max}(|\lambda_2|,|\lambda_3|)$. The base cases of induction are proved in Section \ref{basic}. Assume that $|\lambda_2|=|\lambda_3|+1$. Let $\lambda_2' \in \mathcal{Y}_{r,s}$ be such that 
\begin{enumerate}
\item $\lambda_2$ is obtained by adding two boxes from $\lambda_2'$,
\item  $\lambda_3$ is obtained by adding a box to $\lambda_2'$. 
( The other case $|\lambda_3|=|\lambda_2|+1$ is handled similarly.)
\end{enumerate}

\subsubsection{Step IV}We use gauge symmetry (see Section 2) to reduce to the case for the admissible pair $((\omega_1,\lambda_2',\lambda_3), (\omega_1, \lambda_2'^T, \lambda_3^T))$. This is done in Proposition \ref{n=31}. Now $\operatorname{max}(|\lambda_2'|,|\lambda_3|)<|\lambda_2|$. The other case is handled similarly. Hence, we are done by induction.

\begin{remark}The minimal cases here are similar to the minimal cases in ~\cite{A}. In the case of symplectic rank-level duality, T. Abe (see \cite{A}) identified the rank-level duality map with the symplectic strange duality map and used the geometry of parabolic bundles with a symplectic form to show that rank-level duality maps are non-zero. As remarked earlier, we were not able to describe the map in Theorem ~\ref{duality} geometrically. However, the steps described above can be used to tackle minimal cases in ~\cite{A}. Similarly, one can use the same strategy to reprove the rank-level duality results in \cite{BP, NT}. 
\end{remark}
 
\subsubsection{The base cases for induction}\label{basic}
We think of $\mathbb{P}^1$ as $\mathbb{C}\cup\{ \infty\}$ and let $z$ be a global coordinate of $\mathbb{C}$. We will assume that $(P_1,P_2, P_3)=(1,0,\infty)$ with coordinates $\xi_1=z-1$, $\xi_2=z$ and $\xi_3=\frac{1}{z}$ respectively, and denote the associated data by $\mathfrak{X}$.
%Throughout this section we assume that $P_1=1$ with coordinate $\xi_1=z-1$, $P_2=0$ with coordinate $\xi_2=z$, $P_3=\infty$ with coordinate $\xi_3=\frac{1}{z}$ and let $\mathfrak{X}$ be the data associated to $\mathbb{P}^1$ and points $\vec{p}=(P_1,P_2,P_3)$ with chosen coordinates. 
Further we let $\vec{\Lambda}=(\omega_1,\omega_1,0)$.

\begin{lemma}\label{01}
Let $\vec{\lambda}=(\omega_1,\omega_1,0)$. Then, the following map 
$$\mathcal{V}_{\vec{\lambda}}(\mathfrak{X},\mathfrak{so}(2r+1),2s+1)\otimes \mathcal{V}_{\vec{\lambda}^T}(\mathfrak{X},\mathfrak{so}(2s+1),2r+1) \rightarrow \mathcal{V}_{\vec{\Lambda}}(\mathfrak{X},\mathfrak{so}(2d+1),1),$$ is non-zero.
\end{lemma}	
\begin{proof}
Let $\langle {\Psi'}| \in \mathcal{V}^{\dagger}_{\vec{\Lambda}}(\mathfrak{X},\mathfrak{so}(2d+1),1)$ be a non-zero element and $\langle \Psi|$ is the image of $\langle {\Psi'}|$ under the propagation of vacua. It is enough to produce $|\Phi_1\otimes \Phi_2 \otimes \Phi_3\rangle  \in \mathcal{H}_{\vec{\lambda}}(\mathfrak{so}(2r+1))\otimes \mathcal{H}_{\vec{\lambda}^T}(\mathfrak{so}(2s+1))$ such that 
$$\langle {\Psi'} | \Phi_1\otimes \Phi_2 \otimes \Phi_3\rangle \neq 0.$$
We choose $|\Phi_1\rangle =v^{-1,-1}$, $|\Phi_2\rangle =v^{1,1}$ and $|\Phi_3\rangle =1$. By propagation of vacua, we get 
\begin{eqnarray*}
\langle \Psi'| v^{-1,-1}\otimes v^{1,1}\otimes 1\rangle &=&\langle \Psi| v^{-1,-1}\otimes v^{1,1}\rangle,\\
&=&Q(v^{-1,-1},v^{1,1}),\\
&=&1.
\end{eqnarray*}
\end{proof}
\begin{lemma}\label{12}
Let $\vec{\lambda}=(\omega_1,\omega_1, 2\omega_1)$ or $(\omega_1,\omega_1,\omega_2)$. Then, the following map 
$$\mathcal{V}_{\vec{\lambda}}(\mathfrak{X},\mathfrak{so}(2r+1),2s+1)\otimes \mathcal{V}_{\vec{\lambda}^T}(\mathfrak{X},\mathfrak{so}(2s+1),2r+1) \rightarrow \mathcal{V}_{\vec{\Lambda}}(\mathfrak{X},\mathfrak{so}(2d+1),1),$$ is non-zero.
\end{lemma}
\begin{proof}
First let $\vec{\lambda}=(\omega_1,\omega_1, 2\omega_1)$. We choose $|\Phi_3\rangle $ to be the opposite highest weight vector of the module $\mathcal{H}_{2\omega_1}(\mathfrak{so}(2r+1))\otimes \mathcal{H}_{\omega_2}(\mathfrak{so}(2s+1))$, $|\Phi_2\rangle = v^{1,1}$. We choose $|\Phi_1\rangle$ such that the $\mathfrak{H}$-weight of $|\Phi_1\otimes \Phi_2\otimes \Phi_3\rangle $ is zero. In this case, $|\Phi_1\rangle = v^{1,2}$. By gauge symmetry, we get the following:
\begin{eqnarray*} 
&&\langle \Psi'| v^{1,2}\otimes v^{1,1}\otimes B^{-1,-1}_{1,2}(-1).1\rangle\\
&& =\langle \Psi'| v^{1,2}\otimes v^{1,1}\otimes B^{-1,-1}_{1,2}(\frac{1}{\xi_3}).1\rangle,\\
&& =-\langle \Psi'| B^{-1,-1}_{1,2}v^{1,2}\otimes v^{1,1}\otimes1\rangle -\langle \Psi'| v^{1,2}\otimes B^{-1,-1}_{1,2}(1).v^{1,1} \otimes1\rangle,\\
&&=-\langle \Psi'| v^{-1,-1}\otimes v^{1,1}\otimes 1\rangle \ \  \mbox{[Since $B^{-1,-1}_{1,2}(1).v^{1,1}=0$]}, \\
&&\neq 0.\ \ \mbox{[By Lemma \ref{01}]}
\end{eqnarray*} 
The case $\vec{\lambda}=(\omega_1,\omega_1,\omega_2)$ follows similarly.

\end{proof}
%\begin{remark} The proof when $\vec{\lambda}=(\omega_1,\omega_1,\omega_2)$ is similar to Proposition ~\ref{12}.
%\end{remark}
\begin{lemma}\label{23}Let $\vec{\lambda}=(\omega_1,\omega_1+\omega_2, 2\omega_1)$ or $(\omega_1,\omega_1+\omega_2,\omega_2)$. Then, the following map: 
$$\mathcal{V}_{\vec{\lambda}}(\mathfrak{X},\mathfrak{so}(2r+1),2s+1)\otimes \mathcal{V}_{\vec{\lambda}^T}(\mathfrak{X},\mathfrak{so}(2s+1),2r+1) \rightarrow \mathcal{V}_{\vec{\Lambda}}(\mathfrak{X},\mathfrak{so}(2d+1),1),$$ is non-zero.  
\end{lemma}
\begin{proof}
Consider $\lambda'_2=\omega_1$ and $\lambda_2=\omega_1+\omega_2$ and let $\lambda_2$ is obtained from $\lambda'_2$ by adding two boxes in the $(1,2)$ and $(2,1)$ coordinate. Thus, by Proposition \ref{Kacmoody}, we get 
$$v_{\lambda_2}=B^{1,2}_{-2,-1}(-1)v^{1,1}.$$ 

As in Lemma \ref{12}, the vector $|\Phi_3\rangle = B^{-1,-1}_{1,2}(-1).1$. We choose $|\Phi_2\rangle =v_{\lambda_2}$ and $|\Phi_1 \rangle $ such that the $\mathfrak{H}$-weight of $|\Phi_1\otimes \Phi_2\otimes \Phi_2\rangle $ is zero. In this case, $|\Phi_1\rangle =v^{-2,-1}$. By gauge symmetry, we get the following:

\begin{eqnarray*} 
&&\langle \Psi'| v^{-2,-1}\otimes B^{1,2}_{-2,-1}(-1)v^{1,1}\otimes B^{-1,-1}_{1,2}(-1).1\rangle\\
&&\ \ \  =-\langle \Psi'| B^{1,2}_{-2,-1}v^{-2,-1}\otimes v^{1,1}\otimes B^{-1,-1}_{1,2}(-1).1\rangle\\
&& \ \ \ \ \ \ \  \ -\langle \Psi'| v^{-2,-1}\otimes v^{1,1}\otimes B^{1,2}_{-2,-1}(1) B^{-1,-1}_{1,2}(-1).1\rangle, \\
&&\ \ \  =-\langle \Psi'| B^{1,2}_{-2,-1}v^{-2,-1}\otimes v^{1,1}\otimes B^{-1,-1}_{1,2}(-1).1\rangle\\
&&\ \ \ \ \ \ \ \ \ -\langle \Psi'| v^{-2,-1}\otimes v^{1,1}\otimes B^{-1,-1}_{1,2}(-1)B^{1,2}_{-2,-1}(1) .1\rangle\\
&& \ \ \ \ \ \ \ \ \ \ -\langle \Psi'| v^{-2,-1}\otimes v^{1,1}\otimes [B^{1,2}_{-2,-1}(1),B^{-1,-1}_{1,2}(-1)] .1\rangle, \\
&&\ \ \ =-\langle \Psi'| v^{1,2}\otimes v^{1,1}\otimes B^{-1,-1}_{1,2}(-1).1\rangle,\\
&&\ \ \ \neq 0.\ \ \mbox{[By Lemma \ref{12}]}
\end{eqnarray*}

\end{proof}

\subsubsection{The inductive step}
\begin{proposition}\label{n=31}
Let $|\lambda_2|=|\lambda_3|+1$ and $\lambda_2$ be obtained from $\lambda_3$ by adding a box in the $(c,d)$-th coordinate. Further, assume that $\lambda_3$ is obtained from $\lambda_2'$ by adding a box in the $(a,b)$-th coordinate. Then, the rank-level duality isomorphism for the admissible pair $((\omega_1,\lambda_2', \lambda_3),(\omega_1, \lambda_2'^T,\lambda_3^T))$ implies rank-level duality isomorphism for the following admissible pair $((\omega_1,\lambda_2, \lambda_3),(\omega_1, \lambda_2^T,\lambda_3^T))$.
\end{proposition}

\begin{proof}Without loss of generality assume that $(a,b)<(c,d)$. Consider a non-zero element $\langle \Psi'| \in \mathcal{V}_{\vec{\Lambda}}^{\dagger}(\mathfrak{X},\mathfrak{so}(2d+1),1)$. We choose $|\Phi_1\rangle=v^{-a,-b}$, $|\Phi_2\rangle =B^{a,b}_{-c,-d}(-1)v_{\lambda_2'}$ and $|\Phi_3\rangle$ to be the opposite highest weight vector of the component with highest weight $(\lambda_3, \lambda_3^T)$. Then, we have the following: 
\begin{eqnarray*}
&&\langle \Psi' |\Phi_1\otimes \Phi_2\otimes \Phi_3\rangle \\
&&=\langle \Psi' | v^{-a,-b}\otimes B^{a,b}_{-c,-d}(-1)v_{\lambda_2'}\otimes \Phi_3\rangle, \\
&&=-\langle \Psi' | B^{a,b}_{-c,-d}v^{-a,-b}\otimes v_{\lambda_2'}\otimes \Phi_3\rangle \\
&& \ \ \ \ \ \ -\langle \Psi' | v^{-a,-b}\otimes v_{\lambda_2'}\otimes B^{a,b}_{-c,-d}(1)\Phi_3\rangle, \\
&&=\langle \Psi' | v^{c,d}\otimes v_{\lambda_2'}\otimes \Phi_3\rangle. \ \ \text{( By Lemma ~\ref{vanish1} below)}
\end{eqnarray*}
The last expression is exactly the one that we consider to prove the rank-level duality for the admissible pair $((\omega_1,\lambda_2', \lambda_3),(\omega_1, \lambda_2'^T,\lambda_3^T))$. Hence we are done.
\end{proof}
\begin{lemma}\label{vanish1}With the above notation, we have the following:
 $$B^{a,b}_{-c,-d}(1)|\Phi_3\rangle=0. $$
\end{lemma}
\begin{proof}Since $|\lambda_3|$ is even, the opposite highest weight vector $|\Phi_3\rangle$ can be chosen to be of the form $B^{-a,-b}_{e,f}(-1)v$. Moreover $v$ has the form $\prod_{\alpha \in I} X_{-\alpha}(-1).1$ such that $(L_{a,b},\alpha)=0$, where $I$ is a subset of positive root of $\mathfrak{so}(2d+1)$ and $X_{-\alpha}$ is a non-zero element in the weight space of the negative root $-\alpha$.
\begin{eqnarray*}
B^{a,b}_{-c,-d}(1)|\Phi_3\rangle&=&B^{a,b}_{-c,-d}(1)B^{-a,-b}_{e,f}(-1)v,\\
&=&B^{-a,-b}_{e,f}(-1)B^{a,b}_{-c,-d}(1)v + [B^{a,b}_{-c,-d}(1), B^{-a,-b}_{e,f}(-1)]v,\\
&=& B^{-a,-b}_{e,f}(-1)B^{a,b}_{-c,-d}(1)\prod_{\alpha \in I}X_{-\alpha}(-1).1 + [B^{a,b}_{-c,-d}, B^{-a,-b}_{e,f}]\prod_{\alpha \in I}X_{-\alpha}(-1).1,\\
&=& B^{-a,-b}_{e,f}(-1)\big(\prod_{\alpha \in I}X_{-\alpha}(-1)\big)B^{a,b}_{-c,-d}(1).1 \\
&& + \big(\prod_{\alpha \in I}X_{-\alpha}(-1)\big)[B^{a,b}_{-c,-d}, B^{-a,-b}_{e,f}].1,\\
&=&0.
\end{eqnarray*}
 Hence the lemma follows.

\end{proof}
The proof of the following proposition is similar to the proof of Proposition 9.6 and tackles the case $|\lambda_3|=|\lambda_2|+1$.
\begin{proposition}
Let $|\lambda_3|=|\lambda_2|+1$ and $\lambda_3$ is obtained from $\lambda_2$ by adding a box in the $(c,d)$-th coordinate. Further, assume that $\lambda_2$ is obtained from $\lambda_3'$ by adding a box in the $(a,b)$-th coordinate. Then, the rank-level duality isomorphism for the admissible pair $((\omega_1,\lambda_2, \lambda_3'),(\omega_1, \lambda_2^T,\lambda_3'^T))$ implies rank-level duality isomorphism for the following admissible pair $((\omega_1,\lambda_2, \lambda_3),(\omega_1, \lambda_2^T,\lambda_3^T))$.
\end{proposition}
\subsection{The remaining three point cases}\label{reminimalcases}As before, we will assume that $(P_1,P_2, P_3)=(1,0,\infty)$ with coordinates $\xi_1=z-1$, $\xi_2=z$ and $\xi_3=\frac{1}{z}$. We denote the associated data by $\mathfrak{X}$. Let $\vec{\lambda}=(\omega_1,\lambda,\lambda)$, $\vec{\Lambda}=(\omega_1,\omega_1,0)$, where $\lambda \in \mathcal{Y}_{r,s}$ such that $(\lambda, L_r) \neq 0$. The proof of the next proposition follows the same pattern as the proof of  Proposition \ref{n=31}. We give a proof of the first part of the proposition for completeness.

\begin{proposition}
 The following maps are non-zero:
\begin{enumerate} 
\item  $\mathcal{V}_{\vec{\lambda}}(\mathfrak{X},\mathfrak{so}(2r+1),2s+1)\otimes \mathcal{V}_{\vec{\lambda}^T}(\mathfrak{X},\mathfrak{so}(2s+1),2r+1)\rightarrow \mathcal{V}_{\vec{\Lambda}}(\mathfrak{X},\mathfrak{so}(2d+1),1),$\\ where $|\lambda|$ is odd and $\vec{\lambda}^T=(\omega_1, \lambda^T,\sigma(\lambda^T))$.\\
\item $\mathcal{V}_{\vec{\lambda}}(\mathfrak{X},\mathfrak{so}(2r+1),2s+1)\otimes \mathcal{V}_{\vec{\lambda}^T}(\mathfrak{X},\mathfrak{so}(2s+1),2r+1)\rightarrow \mathcal{V}_{\vec{\Lambda}}(\mathfrak{X},\mathfrak{so}(2d+1),1),$ \\ where $|\lambda|$ is even and $\vec{\lambda}^T=(\omega_1, \sigma(\lambda^T),\lambda^T)$.
\end{enumerate}

\end{proposition}
\begin{proof}
Let $\lambda'\in \mathcal{Y}_{r,s}$ be such that $\sigma(\lambda)$ is obtained by adding boxes in $(0,1)$ and $(r,a)$ to $\lambda'$ and $\lambda$ is obtained by adding a box in the $(r,a)$-th position. Since $|\lambda|$ is odd, the module with highest weight $(\lambda,\sigma(\lambda^T))$ appears in the branching of $\mathcal{H}_{0}(\mathfrak{so}(2d+1))$. By Proposition \ref{highestweight}, the opposite highest weight vector is given by 
$B^{0,-1}_{r,a}(-1)v^{\lambda'}$, where $v^{\lambda'}$ is the opposite highest weight vector of the irreducible module with highest weight $(\lambda',\lambda'^T)$.

As before, we choose $|\Phi_3\rangle$ to be the opposite highest weight vector of the module with highest weight $(\lambda,\sigma(\lambda^T))$. We set $|\Phi_2\rangle $ to be the highest weight vector $v_{\lambda}$ and $|\Phi_1\rangle$ to be such that the $\mathfrak{H}$-weight of $|\Phi_1\otimes \Phi_2\otimes \Phi_3\rangle $ is zero. In this case $|\Phi_1\rangle $ is $v^{0,1}$.

Let $\langle \Psi'| \in \mathcal{V}^{\dagger}_{\vec{\Lambda}}(\mathfrak{X},\frg,1)$ be a non-zero element. We use gauge symmetry as before to get the following:
\begin{eqnarray*}
&&\langle \Psi'|\Phi_1\otimes \Phi_2\otimes \Phi_3\rangle \\
&&=\langle \Psi'| v^{0,1}\otimes v_{\lambda}\otimes B^{0,-1}_{r,a}(-1)v^{\lambda'}\rangle,\\
&&=-\langle \Psi'| B^{0,-1}_{r,a}(-1)v^{0,1}\otimes v_{\lambda}\otimes v^{\lambda'}\rangle  \\
&&\ \ \ \ \ -\langle \Psi'| v^{0,1}\otimes B^{0,-1}_{r,a}(1) v_{\lambda}\otimes v^{\lambda'}\rangle, \\
&&= \langle \Psi'| v^{-r,-a}\otimes v_{\lambda} \otimes v^{\lambda'}\rangle.\ \ \mbox { (By Lemma similar to \ref{vanish1})}
\end{eqnarray*}

Now, we know that $\langle \Psi'| v^{-r,-a}\otimes v_{\lambda} \otimes v^{\lambda'}\rangle\neq 0$, since rank-level duality holds for the admissible pair $((\omega_1,\lambda, \lambda'), (\omega_1,\lambda^T, \lambda'^T))$. This completes the proof.

\end{proof}

\subsection{The proof in the general case}\label{generalal}
In this section, we finish the proof of Theorem \ref{duality}. We now formulate and prove a key degeneration result using the compatibility of rank-level duality and factorization discussed earlier. Let $\vec{\lambda}_1, \vec{\lambda}_2$ be $n_1$, $n_2$ tuples of weights in $P^0_{2s+1}(\mathfrak{so}(2r+1))$. Consider an $n=n_1+n_2$ tuple $\vec{\lambda}=(\vec{\lambda}_1, \vec{\lambda}_2)$ of weights in $P^0_{2s+1}(\mathfrak{so}(2r+1))$. Similarly, consider $\vec{\mu}=(\vec{\mu}_1, \vec{\mu}_2)$ an $(n_1+n_2)$ tuple of weights in $P^0_{2r+1}(\mathfrak{so}(2r+1))$ such that $(\vec{\lambda}, \vec{\mu})$ is an admissible pair.
\begin{proposition}\label{geometricinput} With the above notation, the following statements are equivalent:
\begin{enumerate}

\item The rank-level duality map for the admissible pair $(\vec{\lambda}, \vec{\mu})$ is an isomorphism for conformal blocks on $\mathbb{P}^1$ with $n$ marked points.
 
\item The following rank-level duality maps are isomorphic:
\begin{itemize}
\item Rank-level duality maps are isomorphisms for all admissible pairs of the form $(\vec{\lambda}_1\cup \lambda, \vec{\mu}_1\cup \mu)$ for conformal blocks on $\mathbb{P}^1$ with $(n_1+1)$ marked points. 
\item  Rank-level duality maps are isomorphisms for all admissible pairs of the form $(\lambda \cup \vec{\lambda}_2, \mu \cup\vec{\mu}_2)$ for conformal blocks on $\mathbb{P}^1$ with $(n_2+1)$ marked points. 
\end{itemize}
\end{enumerate}
\end{proposition}

%\begin{proof}
%The proof follows from Proposition \ref{keydegen}, Lemma \ref{keygeometry} and Proposition \ref{flatness}. 
%\end{proof}
We first start with a lemma. We give a proof of Proposition \ref{geometricinput} using Lemma \ref{keygeometry}. Let $\mathcal{B}=\operatorname{Spec}\mathbb{C}[[t]]$. Suppose $\mathcal{V}$ and $\mathcal{W}$ are vector bundles on $\mathcal{B}$ of same rank and let $\mathcal{L}$ be a line bundle on $\mathcal{B}$. Consider a bilinear map $f: \mathcal{V}\otimes \mathcal{W} \rightarrow \mathcal{L}$. Assume that on $\mathcal{B}$, there  are isomorphisms
\begin{eqnarray*}
\oplus {s_i}:& \mathcal{V} \rightarrow &  \bigoplus_{i\in I} \mathcal{V}_i, \\ 
\oplus t_j :&  \mathcal{W}  \rightarrow &  \bigoplus_{j\in I}\mathcal{W}_j.
\end{eqnarray*}
Further assume that $\mathcal{V}_i$ and $\mathcal{W}_i$ have the same rank. Let $f_{i,j}$ be maps from $\mathcal{V}_i\otimes \mathcal{W}_j \rightarrow {\mathcal{L}}$ such that $f_{i,j}=0$ for $i\neq j$ and $f=\sum_{i\in I} t^{m_i}(f_{i,i}\circ (s_i\otimes t_i))$.
The following lemma is easy to prove. 
\begin{lemma}\label{keygeometry} 
The map $f$ is non-degenerate on $\mathcal{B}^*=\mathcal{B}\setminus \{ t=0\}$ if and only if for all $i\in I$ the maps $f_{i,i}$'s are non-degenerate.

\end{lemma}
\subsubsection{Proof of Proposition \ref{geometricinput}}
We now return to the proof of Proposition \ref{geometricinput}. Let $\mathcal{X}\rightarrow \mathcal{B}$ be a family of curves of genus $0$ such that the generic fiber is a smooth curve and the special fiber $\mathcal{X}_0$ is a nodal curve. In our case, we let $\mathcal{V}$, $\mathcal{W}$  and $\mathcal{L}$ be locally free sheaves $\mathcal{V}_{\vec{\lambda}}(\mathcal{X}, \mathfrak{so}(2r+1), 2s+1)$ and $\mathcal{V}_{\vec{\mu}}(\mathcal{X}, \mathfrak{so}(2s+1), 2r+1)$ and $\mathcal{V}_{\vec{\Lambda}}(\mathcal{X}, \mathfrak{so}(2d+1), 1)$ respectively, where $\vec{\lambda}$ and $\vec{\mu}$ as in Proposition \ref{geometricinput} and $\vec{\Lambda} \in (P^0_{1}(\mathfrak{so}(2d+1))^n$ be such that $(\vec{\lambda}, \vec{\mu})\in B(\vec{\Lambda})$. 

We consider $\mathcal{V}_i$'s to be locally free sheaves of the form $$\mathcal{V}_{\vec{\lambda}_1\cup\lambda}(\mathfrak{X}_1,\mathfrak{so}(2r+1),2s+1)\otimes \mathcal{V}_{\lambda\cup\vec{\lambda}_2}(\mathfrak{X}_2, \mathfrak{so}(2r+1), 2s+1)\otimes \mathbb{C}[[t]],$$ where $\lambda \in P^0_{2s+1}(\mathfrak{so}(2r+1))$, $\mathfrak{X}_1$, $\mathfrak{X}_2$ be the data associated to disjoint copies of $\mathbb{P}^1$ (which are obtained from normalization of $\mathcal{X}_0$) with $n_1+1$, $n_2+1$ points respectively. Similarly, we let $\mathcal{W}_j$'s to be locally free sheaves of the form 
$$\mathcal{V}_{\vec{\mu}_1\cup\mu}(\mathfrak{X}_1,\mathfrak{so}(2s+1),2r+1)\otimes \mathcal{V}_{\mu \cup\vec{\mu}_2}(\mathfrak{X}_2, \mathfrak{so}(2s+1), 2r+1)\otimes \mathbb{C}[[t]],$$ where $\mu \in P^0_{2r+1}(\mathfrak{so}(2s+1))$.  

Since there are bijections (the bijections depend on the factorization of $\mathcal{V}_{\vec{\Lambda}}(\mathcal{X}, \mathfrak{so}(2d+1),1)$ into $n_1$ and $n_2$ parts) between $P^0_{2s+1}(\mathfrak{so}(2r+1))$ and $P^0_{2r+1}(\mathfrak{so}(2s+1))$, we can choose the indexing set $I$ in Lemma \ref{keygeometry} to be $Y_{r,s}\sqcup \sigma(Y_{r,s})$. It is also important to point out that $f_{i,j}=0$ for $i\neq j$ is guaranteed by the fact that given $\lambda \in P^0_{2s+1}(\mathfrak{so}(2r+1))$, $\Lambda \in P^0_{1}(\mathfrak{so}(2d+1))$, there exists exactly one $\mu \in P^0_{2r+1}(\mathfrak{so}(2s+1))$ such that $(\lambda, \mu) \in B(\Lambda)$. The proof of Proposition \ref{geometricinput} now follows from Proposition \ref{keydegen}, Lemma \ref{keygeometry} and Proposition \ref{flatness}. 

\begin{remark}
The situation in Proposition \ref{geometricinput} should be compared to Proposition 5.2 in ~\cite{Pau}.
\end{remark} 
 
An immediate corollary of Proposition \ref{geometricinput} is the following:
\begin{corollary}\label{3pointcase}
If rank-level duality holds for $\mathbb{P}^1$ with three marked points, then it holds for $\mathbb{P}^1$ with an arbitrary number of marked points.
\end{corollary} 

By Proposition ~\ref{diaaut2}, we can further reduce to prove the rank-level duality for an admissible pair of the form $((\lambda_1, \lambda, \lambda_2), (\lambda_1^T,\beta, \lambda_2^T))$, where $\lambda_1, \lambda_2 \in \mathcal{Y}_{r,s}$, $\lambda \in P^0_{2r+1}(\mathfrak{so}(2r+1))$ and $\beta \in P^0_{2s+1}(\mathfrak{so}(2s+1))$. Let $\vec{\lambda}=(\omega_1,\dots, \omega_1, \lambda, \lambda_2)$ and $\vec{\mu}=(\omega_1,\dots,\omega_1,\beta, \lambda_2^T)$, the number of $\omega_1$'s is $|\lambda_1|$. Clearly the pair $(\vec{\lambda}, \vec{\mu})$ is admissible. The following corollary is a direct consequence of Proposition \ref{geometricinput} and Lemma \ref{nonzero}.
\begin{corollary}\label{specialcase} Let $\lambda_1, \lambda_2 \in \mathcal{Y}_{r,s}$. If the rank-level duality is an isomorphism for any  $\mathbb{P}^1$ with $|\lambda_1|+2$ marked points for the admissible pair $\vec{\lambda}=(\omega_1,\dots, \omega_1,\lambda, \lambda_2)$ and $\vec{\mu}=(\omega_1,\dots,\omega_1, \beta, \lambda_2^T)$, then the rank-level duality on $\mathbb{P}^1$ is also an isomorphism for the admissible pair $((\lambda_1, \lambda, \lambda_2),(\lambda_1^T, \beta, \lambda_2^T))$.
\end{corollary}

\begin{lemma}\label{nonzero}
Let $\lambda \in \mathcal{Y}_{r,s}$, and $\vec{\lambda}=(\lambda, \omega_1,\dots, \omega_1)$, where the number of $\omega_1$ is $|\lambda|$, then 
 $$\dim_{\mathbb{C}}\mathcal{V}^{\dagger}_{\vec{\lambda}}(\mathfrak{X}, \mathfrak{so}(2r+1), 2s+1)\neq 0.$$
\end{lemma}
\begin{proof}
The proof follows directly by factorization of fusion coefficients and induction on $|\lambda|$. 
\end{proof}

\subsubsection{Reduction to the one dimensional cases}In the previous section, we reduced Theorem \ref{duality} for admissible pairs of the form $\vec{\lambda}=(\omega_1,\dots, \omega_1, \lambda, \lambda_2)$ and $\vec{\mu}=(\omega_1,\dots,\omega_1, \beta, \lambda_2^T)$, where $\lambda \in P^0_{2s+1}(\mathfrak{so}(2r+1))$, the number of $\omega_1$'s are $|\lambda_1|$, $\lambda_2 \in \mathcal{Y}_{r,s}$ and $\beta \in P^0_{2r+1}(\mathfrak{so}(2s+1))$. The following lemma shows that we can further reduce to the case for certain one dimensional conformal blocks on $\mathbb{P}^1$ with three marked points.
\begin{lemma}\label{31omega1}
Let $\lambda_1, \lambda_2 \in P^0_{2s+1}(\mathfrak{so}(2r+1))$ and $\beta_1, \beta_2 \in P^0_{2r+1}(\mathfrak{so}(2s+1))$. If the rank-level duality holds for admissible pairs of the form $((\lambda_1, \omega_1, \lambda_2),(\beta_1,\omega_1, \beta_2))$, then the rank-level duality holds for admissible pairs on $\mathbb{P}^1$ with arbitrary number of marked points.
\end{lemma}

\begin{proof} The proof follows from Proposition \ref{geometricinput}.
%Follows from $n=2$ using the Proposition ~\ref{n=2} and Propositions ~\ref{diaaut1}  and ~\ref{diaaut2}.
\end{proof}
We use Proposition \ref{diaaut2} and Proposition \ref{1dima} to further reduce to the following admissible pairs for certain one dimensional conformal blocks on $\mathbb{P}^1$ with three marked points:
\begin{enumerate}
\item $(\omega_1, \lambda_2, \lambda_3), (\omega, \lambda_2^T, \lambda_3^T)$, where $\lambda_2, \lambda_3 \in \mathcal{Y}_{r,s}$ and $\lambda_2$ is obtained by $\lambda_3$ either by adding or deleting a box. 
\item $(\omega_1, \lambda, \lambda), (\omega_1, \lambda^T, \sigma(\lambda^T))$, where $\lambda \in \mathcal{Y}_{r,s}$ and $(\lambda, L_r)\neq 0$. 
\end{enumerate} 

The rank-level duality in these cases has been proved in Section \ref{minimalcases} and Section \ref{reminimalcases}. This completes the proof of Theorem \ref{duality}.

\section{Key lemmas}

%\subsection{Key Lemmas 1}
\begin{lemma}
Let $\xi =\exp(\frac{\pi\sqrt{-1}}{2(r+s)})$. Consider the matrix $W$ whose $(i,j)$-th entry is the complex number $(\xi^{i(2j-1)}-\xi^{-i(2j-1)})$. Then, 
\[WW^T=\left(\begin{array}{cccc}
c & & &\\
&\ddots & &\\
&&c&\\
& & & 2c
\end{array}\right),\] where $c=-2(r+s)$.
\end{lemma}

%\begin{lemma}
Let $U$ be a partition of $\{1,\dots, r+s\}$ such that $r+s\in U$ and $|U|=r$. Let $P$ be the permutation matrix associated to the permutation $(U,U^c)$. Then,
\[PWW^TP^{-1}=\left(\begin{array}{cccc}
2c & & & \\
&   c& &\\
&&\ddots  & \\
& & & c
\end{array}\right).\]
%\end{lemma}
Let $A=W$ and $B=W^T$ and $U$, $T$ as in Lemma ~\ref{surprise}. Then, we have the following:
\begin{equation}\label{surprise2}
c^{s}\det{A}_{U,T}=\operatorname{sgn}(U,U^c)\operatorname{sgn}(T,T^c)\det{A} \det{B}_{T^c,U^c}.
\end{equation}

We now state and prove one of the two key lemmas that we used in the proof of the equality of dimensions of the rank-level duality map in Section 7. 
Let $[r+s]$ denote the set $\{1,2,\dots, r+s\}$. We define the following sets:
\begin{enumerate}
\item Consider $\lambda=(\lambda^{1}\geq \lambda^{2}\geq \dots \geq \lambda^{r}) \in \mathcal{Y}_{r,s}$. We define $\alpha^i=\lambda^{i}+r+1-i$ and $[\alpha]=\{\alpha^1> \alpha^2> \dots > \alpha^r\}$.  
\item Consider the complement $[\beta]=(\beta^1> \beta^2>\dots >\beta^s)$ of $[\alpha]$ in $[r+s]$. We define another set $[\gamma]=(\gamma^1 > \gamma^2 >\dots > \gamma^s)$ where $\gamma^i=((r+s)-(\beta^{(s+1-i)}-\frac{1}{2}))$. 
\item Let $T=(t_1>t_2>\dots >t_r)$ where $t_i=r+1-i$; $T'=(t'_1>t'_2>\dots > t'_s)$, where $t'_i=s+1-i$ and $T^c=(t_1^c>t_2^c>\dots >t_s^c)$ is the complement of $T$ in $[r+s]$.
\item $U=(u_1>u_2>\dots > u_r)$ be a subset of $[r+s]$ of cardinality $r$ such that $r+s \in U$ and $U^c=(u_1^c>u_2^c>\dots >u_s^c)$ be the complement of $U$ in $[r+s]$.
\end{enumerate}
Then, for $\lambda \in \mathcal{Y}_{r,s}$, we can write by the Weyl character formula \eqref{WCF}
\begin{eqnarray*}
\operatorname{Tr}_{V_{\lambda}}(\exp{\pi\sqrt{-1}\frac{\mu +\rho}{r+s}})&=&\frac{\operatorname{det}(\zeta^{u_i(\alpha^j-\frac{1}{2})}-\zeta^{-u_i(\alpha^j-\frac{1}{2})})}{\operatorname{det}(\zeta^{u_i(t_j-\frac{1}{2})}-\zeta^{-u_i(t_j-\frac{1}{2})})}, \\
&=&\frac{\operatorname{det}(\xi^{u_i(2\alpha^j-1)}-\xi^{-u_i(2\alpha^j-1)})}{\operatorname{det}(\xi^{u_i(2t_j-1)}-\xi^{-u_i(2t_j-1)})},
\end{eqnarray*}
where $\mu+\rho=\sum_{i=1}^ru_iL_i$ and $\zeta=\xi^2$ as in Section 7.

For $\lambda^{T}\in \mathcal{Y}_{r,s}$, $\mu'+\rho'=\sum_{i=1}^su^c_iL_i$ and $\rho'$ the Weyl vector of $\mathfrak{so}(2s+1)$, we can write 
\begin{eqnarray*}
\operatorname{Tr}_{V_{\lambda^T}}(\exp{\pi\sqrt{-1}\frac{\mu' +\rho'}{r+s}})&=&\frac{\operatorname{det}(\zeta^{u^c_i(\gamma^j)}-\zeta^{-u^c_i(\gamma^j)})}{\operatorname{det}(\zeta^{u^c_i(t'_j-\frac{1}{2})}-\zeta^{-u^c_i(t'_j-\frac{1}{2})})},\\
&=& \frac{\operatorname{det}(\zeta^{u^c_i((r+s)-(\beta^j-\frac{1}{2}))}-\zeta^{-u^c_i((r+s)-(\beta^j-\frac{1}{2}))})}{{\operatorname{det}(\zeta^{u^c_i((r+s)-(t^c_j-\frac{1}{2}))}-\zeta^{-u^c_i((r+s)-(t^c_j-\frac{1}{2}))})}},\\
&=& \frac{\operatorname{det}(\zeta^{-u^c_i(r+s)}(\zeta^{u^c_i(\beta^j-\frac{1}{2})}-\zeta^{-u^c_i(\beta^j-\frac{1}{2})}))}{\operatorname{det}(\zeta^{-u^c_i(r+s)}(\zeta^{u^c_i(t^c_j-\frac{1}{2})}-\zeta^{-u^c_i(t^c_j-\frac{1}{2})}))},\\
&=&\frac{\operatorname{det}(\zeta^{u^c_i(\beta^j-\frac{1}{2})}-\zeta^{-u^c_i(\beta^j-\frac{1}{2})})}{\operatorname{det}(\zeta^{u^c_i(t^c_j-\frac{1}{2})}-\zeta^{-u^c_i(t^c_j-\frac{1}{2})})},\\
&=&\frac{\operatorname{det}(\xi^{u^c_i(2\beta^j-1)}-\xi^{-u^c_i(2\beta^j-1)})}{\operatorname{det}(\xi^{u^c_i(2t^c_j-1)}-\xi^{-u^c_i(2t^c_j-1)})}, 
\end{eqnarray*}
where $\xi$ and $\zeta$ are as before.
By applying Equation ~\ref{surprise2}, we get the following:
\begin{lemma}\label{dirty1}
$$\operatorname{Tr}_{V_{\lambda}}(\exp{\pi\sqrt{-1}\frac{\mu +\rho}{r+s}})=\frac{\operatorname{sgn}([\alpha],[\beta])}{\operatorname{sgn}(T,T^c)}
\operatorname{Tr}_{V_{\lambda^T}}(\exp{\pi\sqrt{-1}\frac{\mu' +\rho'}{r+s}}).$$
\end{lemma}
The following can be checked by a direct calculation:
\begin{lemma}
$$\operatorname{sgn}([\alpha],[\beta])=(-1)^{\frac{r(r-1)}{2}+\frac{s(s-1)}{2}+ |\lambda|}.$$
$$\operatorname{sgn}(T,T^c)=(-1)^{\frac{r(r-1)}{2}+\frac{s(s-1)}{2}}.$$
\end{lemma}
Thus, we have the following equality: 
\begin{equation*}
\operatorname{Tr}_{V_{\lambda}}(\exp{\pi\sqrt{-1}\frac{\mu +\rho}{r+s}})=(-1)^{|\lambda|}
\operatorname{Tr}_{V_{\lambda^T}}(\exp{\pi\sqrt{-1}\frac{\mu' +\rho'}{r+s}}).
\end{equation*}
%\subsection{Key Lemmas 2}
Let $\xi=\exp{\frac{\pi\sqrt{-1}}{4(r+s)}}$. Then, the following equality holds for any integers $a$ and $b$:
\begin{equation*}
\xi^{(2(r+s)-(2a-1))(2(r+s)-(2b-1))}=(-1)^{(a+b)}\xi^{(2a-1)(2b-1)}.
\end{equation*}
\begin{lemma}
Let $\xi =\exp(\frac{\pi\sqrt{-1}}{4(r+s)})$. Consider the matrix $W$ whose $(i,j)$-th entry is the complex number $(\xi^{(2i-1)(2j-1)}-\xi^{-(2i-1)(2j-1)})$. Then, the following holds:
\[WW^T=\left(\begin{array}{cccc}
c & & &\\
&\ddots & &\\
&&c&\\
& & & c
\end{array}\right),\] where $c=-2(r+s)$.
\end{lemma}
Let $U$ be a partition of $\{1,\dots, r+s\}$ such that $|U|=r$. %Let $P$ be the permutation matrix associated to the permutation $(U,U^c)$.
%\[PWW^TP^{-1}=\left(\begin{array}{cccc}
%c & & & \\
%&   c& &\\
%&&\ddots  & \\
%& & & c
%\end{array}\right)\]
%\end{lemma}
Let $A=W$, $B=W^T$ and $U$, $T$ as in Lemma ~\ref{surprise}. Then, 
\begin{equation}\label{surprise3}
c^{s}\det{A}_{U,T}=\operatorname{sgn}(U,U^c)\operatorname{sgn}(T,T^c)\det{A} \det{B}_{T^c,U^c}.
\end{equation}

Let $U'=(u_1'>u_2'>\dots > u_r')$ be a subset of $[r+s]$ of cardinality $r$, $U'^c=(u'^c_1>\dots >u'^c_s)$ be the complement of $U'$ in $[r+s]$ and $\mu+\rho=\sum_{i=1}^r(u_i'-\frac{1}{2})L_i$. Then, by the Weyl character formula \eqref{WCF}, we can write the following for $\lambda \in \mathcal{Y}_{r,s}$:
\begin{eqnarray*}
\operatorname{Tr}_{V_{\lambda}}(\exp{\pi\sqrt{-1}\frac{\mu +\rho}{r+s}})&=&\frac{\operatorname{det}(\zeta^{(u'_i-\frac{1}{2})(\alpha^j-\frac{1}{2})}-\zeta^{-(u'_i-\frac{1}{2})(\alpha^j-\frac{1}{2})})}{\operatorname{det}(\zeta^{(u'_i-\frac{1}{2})(t_j-\frac{1}{2})}-\zeta^{-(u'_i-\frac{1}{2})(t_j-\frac{1}{2})})}, \\
&=&\frac{\operatorname{det}(\xi^{(2u'_i-1)(2\alpha^j-1)}-\xi^{-(2u'_i-1)(2\alpha^j-1)})}{\operatorname{det}(\xi^{(2u'_i-1)(2t_j-1)}-\xi^{-(2u'_i-1)(2t_j-1)})}.
\end{eqnarray*}

For $\lambda^{T}\in \mathcal{Y}_{r,s}$, $\mu'+\rho'=\sum_{i=1}^s((r+s+\frac{1}{2})-u'^c_i)L_i$ and $\rho'$ be the Weyl vector of $\mathfrak{so}(2s+1)$, we can write the following:
\begin{eqnarray*}
\operatorname{Tr}_{V_{\lambda^T}}(\exp{\pi\sqrt{-1}\frac{\mu' +\rho'}{r+s}})&=& \frac{\operatorname{det}(\zeta^{((r+s)-(u'^c_i-\frac{1}{2}))((r+s)-(\beta^j-\frac{1}{2}))}-\zeta^{((r+s)-(u'^c_i-\frac{1}{2}))((r+s)-(\beta^j-\frac{1}{2}))})}{{\operatorname{det}(\zeta^{((r+s)-(u'^c_i-\frac{1}{2}))((r+s)-(t^c_j-\frac{1}{2}))}-\zeta^{((r+s)-(u'^c_i-\frac{1}{2}))((r+s)-(t^c_j-\frac{1}{2}))})}},\\
&=&\frac{(-1)^{\sum_{i=1}^s(u'^c_i+\beta_i)}}{(-1)^{\sum_{i=1}^s(u'^c_i+t^c_i)}}	\frac{\operatorname{det}(\xi^{(2u'^c_i-1)(2\beta^j-1)}-\xi^{-(2u'^c_i-1)(2\beta^j-1)})}{\operatorname{det}(\xi^{(2u'^c_i-1)(2t^c_j-1)}-\xi^{-(2u'^c_i-1)(2t^c_j-1)})},\\
&=&\frac{(-1)^{\sum_{i=1}^s(\beta_i)}}{(-1)^{\sum_{i=1}^s(t^c_i)}}	\frac{\operatorname{det}(\xi^{(2u'^c_i-1)(2\beta^j-1)}-\xi^{-(2u'^c_i-1)(2\beta^j-1)})}{\operatorname{det}(\xi^{(2u'^c_i-1)(2t^c_j-1)}-\xi^{-(2u'^c_i-1)(2t^c_j-1)})},\\
&=&(-1)^{|\lambda|}\frac{\operatorname{det}(\xi^{(2u'^c_i-1)(2\beta^j-1)}-\xi^{-(2u'^c_i-1)(2\beta^j-1)})}{\operatorname{det}(\xi^{(2u'^c_i-1)(2t^c_j-1)}-\xi^{-(2u'^c_i-1)(2t^c_j-1)})}.
\end{eqnarray*}
From Equation ~\ref{surprise3}, we get the following lemma:

\begin{lemma}\label{dirty2}
$$\operatorname{Tr}_{V_{\lambda}}(\exp{\pi\sqrt{-1}\frac{\mu +\rho}{r+s}})=\operatorname{Tr}_{V_{\lambda^T}}(\exp{\pi\sqrt{-1}\frac{\mu' +\rho'}{r+s}}).$$
\end{lemma}
\subsection{Some Trace calculations}
Let $\zeta=\exp(\frac{\pi\sqrt{-1}}{r+s})$ and $U=(u_1>u_2>\dots>u_r)$ be a subset of $[r+s]$ of cardinality $r$. Then, we have the following: 
\begin{eqnarray*}
\zeta^{(u_i-\frac{1}{2})((2r+1)+s-\frac{1}{2})}-\zeta^{-(u_i-\frac{1}{2})((2r+1)+s-\frac{1}{2})}&=& \zeta^{(u_i-\frac{1}{2})(2(r+s)-(s-\frac{1}{2}))}-\zeta^{-(u_i-\frac{1}{2})(2(r+s)-(s-\frac{1}{2}))},\\
&=& -\big(\zeta^{-(u_i-\frac{1}{2})(s-\frac{1}{2})}-\zeta^{(u_i-\frac{1}{2})(s-\frac{1}{2})}\big),\\
&=& \zeta^{(u_i-\frac{1}{2})(s-\frac{1}{2})}-\zeta^{-(u_i-\frac{1}{2})(s-\frac{1}{2})}.
\end{eqnarray*}
The above calculation and the Weyl character formula gives us the following lemma:
\begin{lemma}\label{trivial1}
Consider the dominant weight $\lambda=(2s+1)\omega_1$ of $\mathfrak{so}(2r+1)$. Let $U=(u_1>u_2>\dots>u_r)$ be a subset of $[r+s]$ of cardinality $r$ and $\mu+\rho=\sum_{i=1}^r(u_i-\frac{1}{2})L_i$. Then, 
$$\operatorname{Tr}_{\lambda}(\exp(\pi\sqrt{-1}\frac{\mu+\rho}{r+s}))=1.$$
\end{lemma}
Let $\zeta=\exp(\frac{\pi\sqrt{-1}}{r+s})$ and $U=(u_1>u_2>\dots>u_r)$ be a subset of $[r+s]$ of cardinality $r$. Then, we have the following:
\begin{eqnarray*}
\zeta^{u_i((2r+1)+s-\frac{1}{2})}-\zeta^{-u_i((2r+1)+s-\frac{1}{2})}&=& \zeta^{u_i(2(r+s)-(s-\frac{1}{2}))}-\zeta^{-u_i(2(r+s)-(s-\frac{1}{2}))},\\
&=& \zeta^{-u_i(s-\frac{1}{2})}-\zeta^{u_i(s-\frac{1}{2})},\\
&=& -\big(\zeta^{u_i(s-\frac{1}{2})}-\zeta^{-u_i(s-\frac{1}{2})}\big).
\end{eqnarray*}
The proof of the next lemma also follows from the above calculation and the Weyl character formula gives us the following lemma. 
\begin{lemma}\label{trivial2}
Consider the dominant weight $\lambda=(2s+1)\omega_1$ of $\mathfrak{so}(2r+1)$. Let $U=(u_1>u_2>\dots>u_r)$ be a subset of $[r+s]$ of cardinality $r$ and $\mu+\rho=\sum_{i=1}^r u_iL_i$. Then, 
$$\operatorname{Tr}_{\lambda}(\exp(\pi\sqrt{-1}\frac{\mu+\rho}{r+s}))=-1.$$

\end{lemma}

\subsection{Some trigonometric functions}\label{trigfunc}
We recall from ~\cite{OW} a family of trigonometric functions which has surprising identities. These identities are fundamental to the reciprocity laws of the Verlinde formula in ~\cite{OW}. 

Consider a positive integer $k$ and let $f_k(r)=4\sin^2(\frac{r\pi}{k}).$ Given a finite set $U=\{u_1,\dots, u_r\}$ of rational numbers, we consider the following functions defined in Section 1 of ~\cite{OW} (where an empty product is deemed to be 1):

$$\mathcal{P}_k(U)=\prod_{1\leq i < j\leq r}\bigg(f_k(u_i-u_j)f_k(u_i+u_j)\bigg) \hspace{1cm} \mathcal{N}_k(U)=\prod_{i=1}^rf_k(u_i),$$
$$\Phi_{k}(U)=\mathcal{P}_k(U)\mathcal{N}_k(U)$$
%\end{eqnarray*}

We use the function $\Phi_{k}(U)$ to rewrite the Verlinde formula in Section 7. The identities of $\Phi_{k}(U)$ are among the key ingredients in the proof of the equality of the dimensions as discussed in Section 7.

\bibliographystyle{plain}
\def\noopsort#1{}
%\begin{thebibliography}{10}

\end{document}